\documentclass{article}
\RequirePackage{etex} %
\usepackage[letterpaper,margin=1in]{geometry}
\usepackage{amsmath,amssymb,amsthm,amsfonts,latexsym,bbm,xspace,graphicx,float,mathtools,mathdots}
\usepackage{braket,caption,ellipsis,xcolor,textcomp}
\usepackage{combelow} %
\usepackage{indentfirst}
\usepackage{booktabs}
\usepackage[hidelinks]{hyperref}
\usepackage[capitalize,noabbrev]{cleveref}
\crefname{ineq}{inequality}{inequalities}
\creflabelformat{ineq}{#2{\upshape(#1)}#3}

\usepackage{enumitem}

\newtheorem{theorem}{Theorem}

\usepackage{chngcntr}
\counterwithin{theorem}{section}

\newtheorem*{namedtheorem}{\theoremname}
\newcommand{\theoremname}{testing}

\newtheorem{lemma}[theorem]{Lemma}

\newtheorem{proposition}[theorem]{Proposition}
\newtheorem{fact}[theorem]{Fact}
\newtheorem{corollary}[theorem]{Corollary}

\newtheorem{assumption}[theorem]{Assumption}

\theoremstyle{definition}

\newtheorem{remark}[theorem]{Remark}

\newcommand{\calE}{\mathcal{E}}
\newcommand{\calF}{\mathcal{F}}
\newcommand{\calG}{\mathcal{G}}

\newcommand{\calL}{\mathcal{L}}

\newcommand{\calP}{\mathcal{P}}

\newcommand{\calX}{\mathcal{X}}

\newcommand{\bp}{\boldsymbol{p}}

\newcommand{\ignore}[1]{}

\usepackage{autonum}

\DeclareMathOperator*{\argmax}{argmax}
\DeclareMathOperator*{\esssup}{esssup}

\newcommand{\iid}{\text{i.i.d.}\xspace}

\newcommand{\mc}[1]{\mathcal{#1}}

\newcommand{\lb}{\left[}
\newcommand{\rb}{\right]}
\newcommand{\lp}{\left(}
\newcommand{\rp}{\right)}
\newcommand{\lbr}{\left\{}
\newcommand{\rbr}{\right\}}
\newcommand{\lv}{\left\lvert}
\newcommand{\rv}{\right\rvert}

\newcommand{\defined}{\coloneqq}

\newcommand{\simiid}{\stackrel{i.i.d.}{\sim}}

\usepackage{pgfplots}
\usepackage{graphicx}
\usepackage{subcaption}
\pgfplotsset{compat=newest}
\pgfplotsset{scaled y ticks=false}
\usepgfplotslibrary{groupplots}
\usepgfplotslibrary{dateplot}
\tikzstyle{every node}=[font=\small]
\pgfplotsset{
    yticklabel style={/pgf/number format/fixed},  
}

\pgfplotsset{compat=1.11,
 /pgfplots/ybar legend/.style={
 /pgfplots/legend image code/.code={
 \draw[##1,/tikz/.cd,yshift=-0.25em]
 (0cm,0cm) rectangle (3pt,0.8em);},
 },
}
\newcommand{\qtext}[1]{\quad\text{#1}\quad}

\title{Classifier-Based Nonparametric Sequential Hypothesis Testing}
\usepackage{authblk}
\author{Chia-Yu Hsu}
\author{Shubhanshu Shekhar}
\affil{EECS Department, University of Michigan}
\affil{\texttt{\{chiayuh, shubhan\}@umich.edu}}

\begin{document}
\date{}
\maketitle

\begin{abstract}
We consider the problem of constructing sequential power-one tests where the null and alternative classes are specified indirectly through historical or offline data. More specifically, given an offline dataset consisting of observations from $L+1$ distributions $\{P_0, P_1, \ldots, P_L\}$, and a new unlabeled data stream $\{X_t: t \geq 1\} \simiid P_\theta$, the goal is to decide between the null $H_0: \theta = 0$, against the alternative $H_1: \theta \in [L]:=\{1,\ldots,L\}$. Our main methodological contribution is a general approach for designing a level-$\alpha$ power-one test for this problem using a multi-class classifier trained on the given offline dataset. 

Working under a mild ``separability'' condition on the distributions and the trained classifier, we obtain an upper bound on the expected stopping time of our proposed level-$\alpha$ test, and then show that in general this cannot be improved. In addition to rejecting the null, we show that our procedure can also identify the true underlying distribution almost surely. We then establish a sufficient condition to ensure the required separability of the classifier, and provide some converse results to investigate the role of the size of the offline dataset and the family of classifiers among classifier-based tests that satisfy the level-$\alpha$ power-one criterion. Finally, we present an extension of our analysis for the training-and-testing distribution mismatch 
and illustrate an application to sequential change detection. Empirical results using both synthetic and real data provide support for our theoretical results.

\end{abstract}

\section{Introduction}
\label{sec:introduction}

The field of sequential hypothesis testing, initiated by~\cite{Wald1945Sequential}, involves designing statistical procedures for determining whether the null or the alternative hypothesis is true through sequentially querying a sequence of testing data. An important advantage of such procedures is that they can adaptively determine the number of queried samples needed to achieve a desired testing accuracy, %
thereby using sampling resources more efficiently than fixed-length tests. Some prior work in this area under the assumption that the null is precisely characterized, either by making strong parametric assumptions on the distributions or through some structural properties~(e.g., independence, symmetry, moment constraints, etc.). 
In this paper, we consider testing problems in which the null and alternatives are specified indirectly through \emph{offline or historical datasets} consisting of $(L+1)$ components drawn from distributions $\{P_0, \ldots, P_L\}$. 
We assume that under the null hypothesis, the test data are generated from $P_0$, while under the alternative hypothesis, we consider a general setting in which the test data are generated from $P_\theta$, where $\theta\in[L]:=\{1,\ldots,L\}$. 
Our goal is to develop a general machinery for designing level-$\alpha$ power-one tests~\cite{Darling1968} using  multiclass classifiers trained on the offline data, and analyze their statistical properties. 
 
This limited knowledge of the ground-truth distribution is motivated by several practical applications, where one cannot specify reliable probabilistic models for the null and alternative hypotheses, but historical data are available. For example, in LLM detection \cite{Chen2025,chen2025_black}, it is difficult to characterize the text-generation distributions of either humans or machines, but one typically has extensive offline datasets consisting of paragraphs generated by humans and by machines. In such settings, classical likelihood-based sequential tests developed following \cite{Wald1945Sequential} are not directly applicable because the ground-truth distributions are unknown; the only information about them comes from the offline dataset. Moreover, due to the nonparametric consideration, rather than studying strong notions of optimality, e.g., minimizing type-I and type-II errors subject to an expected stopping time constraint, which are typically feasible to be developed under parametric assumptions \cite[Section 5.2]{tartakovsky2014sequential}, we focus on a weaker but broadly tractable statistical guarantee: the level-$\alpha$, power-one criterion, initiated by \cite{Darling1968}. Under this criterion, one aims to design a \emph{stopping rule} for a uniform termination error control under the null (the probability of stopping in finite time is constrained by a constant under the null), power-one under any alternative (the probability of finitely stopping is one under the alternative), and sometimes a (minimax/instance-dependent optimal) expected stopping time under the alternative. %

Hypothesis testing based on offline training data was initiated by Gutman~\cite{Gutman89}, but he studied fixed-length testing (the testing sample size is fixed) over a finite-alphabet discrete probabilistic data-generating model. A series of sequential extensions of Gutman’s setting has been developed, starting from \cite{haghifam2021}, where they study strong optimality conditions in the sequential setting from an asymptotic viewpoint (the training sample size goes to infinity). Then, \cite{Hsu2022,Li2025} considered a mixed setting where at least one of the training and testing samples is collected sequentially. Nevertheless, those extensions consider the same probabilistic model as in Gutman’s setting; therefore, they are all parametric. For the nonparametric setting, recent works \cite{Gerber2023,gerber2024likelihoodfreehypothesistesting,Gerber2025} have developed \emph{fixed-length} likelihood-free hypothesis tests for some nonparametric family of distributions. These studies characterized the fundamental tradeoff between training and testing sample sizes \cite{gerber2024likelihoodfreehypothesistesting}, established minimax optimal testing sample complexities \cite{Gerber2025}, and generalized the framework to accommodate testing samples drawn from mixture distributions rather than pure ground-truth distributions \cite{Gerber2023}. However, while they leverage offline-training data to test the unknown labeled data and focus on the standard type-I and type-II error control, their analyses are fundamentally confined to a \emph{fixed-length} regime. They do not address the \emph{variable-length} (i.e., sequential) testing scenario, which we consider in this paper. %
For sequential nonparametric scenarios, \cite{Shekhar2024,Lheritier2015,Hsu2025} investigate sequential two-sample testing, where they assume that both the reference (training) data and the testing data are queried sequentially during the online testing phase. %
Those methods, in our understanding, are more suitable for ``from-scratch'' scenarios where no historical data are available, requiring online training data. 
In contrast, we consider the offline dataset, i.e., the number of training data is fixed. To the best of our knowledge, this is the first work studying the sequential nonparametric hypothesis testing with an offline dataset. We summarize the above discussion in Table~\ref{tab:2x2-cite}.
\begin{table}[t]
\centering
\begin{tabular}{c|cc}
\hline
 & Sequential & Fixed-length \\
\hline
Parametric & \cite{haghifam2021,Hsu2022,Li2025} (\cite{Hsu2022,Li2025} also considered online training data) & \cite{Gutman89} \\
Nonparametric & This paper (\cite{Shekhar2024,Lheritier2015,Hsu2025} considered online references) & \cite{Gerber2023,gerber2024likelihoodfreehypothesistesting,Gerber2025} \\
\hline
\end{tabular}
\caption{Different types of works for hypothesis testing based on offline training data. Column: the dynamic of testing sample size. Row: the type of family of distributions.}
\label{tab:2x2-cite}
\end{table}

To leverage offline data generated without making strong distributional assumptions, motivated by \cite{gerber2024likelihoodfreehypothesistesting}, we propose a \emph{classifier-based} sequential power-one test to bridge this gap. Utilizing a classifier in our design process is motivated the recent advances in applied machine learning~(ML), leading to the existence of powerful ML models designed for complex data. For instance, Fast-DetectGPT \cite{Bao2024} effectively identifies machine-generated text in LLM detection problems, while VGG16~\cite{Simonyan2015} and other CNN-based models demonstrate strong classification capabilities for image data. Therefore, to ensure the testing problem is in general well-posed given a classifier family $\calG$, we restrict our attention to distribution tuples $\{P_0,\ldots,P_L\}$ belonging to a \emph{distinguishable} set specified by $\calG$. Specifically, we require the existence of a \emph{separable} classifier $g \in \calG$ for the tuple: if a data sample (denote as $X$) is drawn from $P_\theta$, then the probability of $g(X)=\theta$ is greater than that of $g(X)=m$ for any $m\in\calL\setminus\{\theta\}$ for any $\theta\in\{0,\ldots,L\}$.

Our proposed stopping rule uses an e-process \cite[\S~4.1]{ramdas2020admissible} constructed by passing the test stream through the classifier. We briefly introduce the construction as follows: We first collect the offline data generated from these distributions and train a multi-classifier $g\in\calG$ that predicts which hypothesis (label) a data sample belongs to.
Then, in the online testing phase, the decision maker sequentially queries one testing sample at each time $t$ (denote as $X_t$), and it is fed into $g$ to produce a one-shot predicted hypothesis (or label) $g(X_t)\in\{0,\ldots,L\}$. We then construct a \emph{test martingale} by turning the label sequence $\{g(X_t)\}_{t\ge1}$ into an e-process via the \emph{testing-by-betting} paradigm \cite{Shafer2019}. Specifically, at each time $t$, we compare whether the current prediction equals the null label versus a running estimate $\hat{j}_t$ of the most frequent predicted label among the past predictions $\{g(X_i)\}_{i=1}^{t-1}$. We use this comparison to define a multiplicative betting update.  
The e-process stops and rejects the null when it exceeds an $\alpha$-dependent threshold. Intuitively, due to the separability provided by the distributions and the classifier, if the null hypothesis is true, $\{g(X_t)\}_{t\ge1}$ concentrate on the null label so that $\hat{j}_t$ tends to indicate the null, and the resulting e-process remains small; in contrast, if the alternative is true, $\{g(X_t)\}_{t\ge1}$ concentrate on an alternative label so that $\hat{j}_t$ tends to indicate the alternative, and the e-process grows quickly and eventually crosses the rejection threshold. With this intuition, we show that our proposed stopping rule satisfies the level-$\alpha$ power-one criterion under a \emph{fixed} distinguishable tuple by Ville's inequality \cite{Ville1939} and martingale concentration bounds. Then, we also provide a non-asymptotic upper bound (for each $\alpha\in(0,1)$) for the expected stopping time that, in general, is asymptotically optimal. Moreover, we demonstrate that the $\hat{j}_t$ will eventually indicate the true index of the underlying distribution almost surely.

The above guarantees are established under the separability condition and ignore how training components (e.g., the size of the training data and the realization of the learning model) affect the resulting classifier-based sequential test. Therefore, we next study (i) how to obtain a classifier that satisfies the separability condition and (ii) how the learning model (training sample size and the family of classifiers) impacts the classifier-based sequential level-$\alpha$ power-one test.

For (i), it is intuitive that if the size of the training samples is larger, then it is more likely to satisfy the separability condition. Specifically, we argue that for any $\{P_0,\ldots,P_L\}$ in the distinguishable set, there exists an empirical risk minimization~(ERM) algorithm that guarantees that the ERM-resulting classifier is also separable for that distribution tuple, provided $\calG$ is PAC-learnable and the training sample size satisfies a sufficient lower bound. 
Next, we address (ii) by establishing \emph{converse} bounds for all \emph{uniformly} achievable tests, i.e., tests satisfying the level-$\alpha$ power-one criterion uniformly across the distinguishable set specified by $\calG$. We first provide a lower bound on the training sample size, revealing that closer underlying distributions naturally demand a larger training sample size of training samples for the uniform achievability. Furthermore, we establish a minimax lower bound on the type-II error counterpart (specifically, the tail probability of the stopping time), which explicitly depends on the training size, the threshold $\alpha$, and the classifier family. This bound yields two key insights: increasing the training sample size reduces the lower bound (better learnability), yet even with infinite training data, the bound remains positive due to the inherent worst-case testing difficulty; and enlarging the classifier family can further reduce the lower bound.

{We further extend our analysis to address practical challenges/applications. First, to model training–testing mismatch motivated by \cite{Pan2010}, we study a distribution-shift setting where the data distribution during the online testing phase mismatches the data distribution for training the classifier applied for the testing phase. We show that the proposed stopping rule remains valid provided that the strength of the mismatch is controlled well, determined by the separability of the classifier for the training dataset.} 
Second, we use our sequential test as a building block to design a sequential change-point detection procedure and analyze its performance.

Finally, we provide empirical results with synthetic and real data to support our theoretical results. In both our basic case and extensions, we argue that our expected stopping time correctly describes the trend of the corresponding empirical behavior. Moreover, we demonstrate that if one can apply multiple-classifiers on making the decision, e.g., consider a mixture of some parallel e-processes, one can obtain a better expected stopping time. That indicates an interesting future work.

\paragraph{Organization} 
The remainder of the paper is organized as follows. In Section~\ref{sec:problem_form}, we formalize our problem. In Section~\ref{sec:main-results}, we present the proposed test, establish its level-$\alpha$, power-one guarantee, and provide its expected stopping time. Then, in Section~\ref{sec:extension}, we illustrate the extensions of our proposed stopping rule. Finally, in Section~\ref{sec:experiments}, we show experimental results that support our theories.

\section{Problem Formulation}
\label{sec:problem_form}
We begin this section by introducing statistical assumptions on the historical data and the test stream. After that we present the definition a general classifier-based sequential test, and conclude the section by describing the key performance measures.  

For some integer $L \geq 1$, let $\mathbf{P} \coloneqq \{P_0, P_1, \ldots, P_L\}$ denote a tuple of data-generating distributions on a general alphabet $\calX$. 
Let $\{X_t\}_{t\ge1} \simiid P_\theta$ denote a stream of $\calX$-valued observations, where $P_\theta$ is an unknown distribution in $\mathbf{P}$. 
We have no further information on $\mathbf{P}$, but only the historical or offline data drawn from this tuple:
\begin{align}
    T_\vartheta^N \coloneqq \{T_{\vartheta, 1}, \ldots, T_{\vartheta, N} \} \simiid P_\vartheta, \qtext{for} \vartheta \in \calL \coloneqq \{0, 1, \ldots, L\}.
\end{align}
Here, for simplicity, we have assumed that we have an equal number of offline data samples~(i.e., $N$) from each distribution in $\mathbf{P}$. Furthermore, we also assume that there is an exact match between the offline and the test distributions; that is, the unknown test distribution $P_\theta$ is exactly one of the $(L+1)$ elements of $\mathbf{P}$ from which the offline data are drawn. As we discuss in Section~\ref{sec:extension}, this  is not strictly necessary, and our proposed test can still work under certain types of distribution shifts. 

Then, our goal is to decide between the null $H_0: \theta = 0$, against the alternative $H_1: \theta \in [L]$. We consider a sequential framework where the decision-maker continuously queries a testing sample at each time slot and dynamically stops the sampling process once sufficient evidence is accumulated. The random time at which sampling terminates is defined as the \emph{stopping time}, denoted by $\tau \in \mathbb{N} \cup \{+\infty\}$. Specifically, the decision to stop is determined by the information available up to that moment, including both the offline dataset and the online testing samples received so far. In words, we denote the natural filtration $\calF_t$ as follows: 
\[
\forall\,t\ge1\quad\mathcal{F}_t \defined \sigma\!\Bigl(\{T^N_\vartheta\}_{\vartheta\in\calL},\, \{X_s\}_{s\le t}\Bigr)\qtext{and} \calF_0 \defined \sigma(\{T^N_\vartheta\}_{\vartheta\in\calL}),
\]
and the stopping time $\tau$ is adapted to $\{\mathcal{F}_t\}_{t\ge0}$.

Formally, our goal is to construct stopping times satisfying the \emph{level-$\alpha$ power-one criterion}, which  mandates that under a given ground-truth tuple $\mathbf{P}$, the probability of stopping in finite time under the null must be less than or equal to $\alpha\in(0,1)$, whereas finite-time stopping under the alternative must occur with probability one. In words, we aim to develop a stopping rule such that the resulting stopping time $\tau$ satisfies $
\mathbb{P}_{\mathbf{P}, 0}(\tau<\infty) \leq \alpha$ and $\mathbb{P}_{\mathbf{P}, \theta}(\tau<\infty) = 1$ for any $\theta\in[L]$,
where $\mathbb{P}_{\mathbf{P}, \theta}(\cdot)$ denotes the probability measure defined on the filtration $\mathcal{F}_\infty \coloneqq \sigma \lp \cup_{t \geq 0} \calF_t \rp$ when the testing samples are $X_t\simiid P_\theta$ (the $\theta$-th distribution in $\mathbf{P}$) for all $t\ge 1$, and the offline training samples are $T^N_\vartheta \simiid P_\vartheta$ for all $\vartheta\in\calL$. Moreover, we denote $\mathbb{E}_{\mathbf{P},\theta}[\cdot]$ as the expectation under the probability measure  $\mathbb{P}_{\mathbf{P}, \theta}$ for each $\theta\in[L]$.

In this paper, we are mainly interested in cases where the alphabet $\calX$ may be complicated~(e.g., images, text, etc), and thus accurate empirical distribution estimation may be computationally prohibitive. Motivated by prior works such as \cite{gerber2024likelihoodfreehypothesistesting, kim2021classification}, we bypass this challenge by introducing a \emph{multi-class classifier}, denoted by $g$ trained on the offline dataset. This classifier serves as a dimensionality reduction mechanism, mapping the intricate high-dimensional space $\calX$ into a significantly more tractable, discrete domain. Specifically, the goal of the classifier in this paper is to learn to distinguish samples from $\mathbf{P}=\{P_0,\ldots,P_L\}$ based on the offline/historical datasets.  We define the learning policy as a function $\phi$ which maps a set of data label pairs $\{T^N_\vartheta,\vartheta\}_{\vartheta\in\calL}$ to a classifier $g$ belonging to a family of classifiers $\calG$, i.e., $\phi(\{T^N_\vartheta,\vartheta\}_{\vartheta\in\calL})=g$ for a $g\in\calG$. We assume that $\calG$ is a hypothesis class of multi-class classifiers mapping from $\calX$ to $\calL$ (a relatively simple set), i.e., $g(x)\in\calL$ for any $x\in\calX$ and $g\in\calG$, and it has a bounded complexity (details in Section~\ref{sec:role_of_clf}). Moreover, we consider a fixed $\calG$ throughout this paper. 

We now introduce the notion of a \emph{distinguishable} distribution family (denoted as $\mathcal{P}_{\text{sep}}$) specified by $\calG$, to which $\mathbf{P}$ belongs. 
Specifically, we consider distributions that are inherently distinguishable by the given classifier family $\calG$. Formally, 
\begin{align}
\mathcal{P}_{\text{sep}}:=\left\{ \mathbf{P}: \exists\, g\in\calG\quad\text{such that}\quad\forall\,\theta\in\calL,\quad \mathbb{P}_{X \sim P_\theta}(g(X)=\theta) > \max_{m \neq \theta} \mathbb{P}_{X \sim P_\theta}(g(X)=m)\right\}.\label{eq:P-sep}
\end{align}
If a tuple $\mathbf{P}\in\mathcal{P}_{\text{sep}}$, we say $\mathbf{P}$ and the corresponding $g  \equiv g(\mathbf{P})$ satisfy the \emph{separability condition} or the classifier $g$ is a \emph{separable classifier} for the $\mathbf{P}$.

With this background, we can now introduce the class of sequential testing procedures that will be studied in this paper.
First, the learning policy $\phi$ outputs a classifier $g\in\calG$ trained on the offline training data $\{T^N_\vartheta\}_{\vartheta\in\calL}$. Then, under this classifier, at each time slot $t$, the decision-maker observes a sample $X_t$ generated \emph{i.i.d.} from an unknown $P_\theta$ where $\theta\in\mathcal{L}$ and inputs them to $g$ and gets $g(X_t)\in\calL$. After obtaining the classifier's output, the decision-maker either stops and declares $\tau=t$, in which case it rejects the null hypothesis, or continues sampling otherwise. Then, we denote $\pi\defined(\phi,\tau)$ as an ``overall'' classifier-based sequential hypothesis test.

\begin{remark}[Permutation]\label{remark:permutation}
    Note that the class $\calP_{\text{sep}}$ is invariant under permutation of the labels; that is, if $\mathbf{P}=\{P_0,\ldots,P_L\} \in \mathcal{P}_{\text{sep}}$, then any tuple $\{P_{\sigma(0)}, P_{\sigma(1)}, \ldots, P_{\sigma(L)}\}$ also belongs to $\mathcal{P}_{\text{sep}}$, where $\sigma$ is any permutation of $\{0, \ldots, L\}$. Moreover, their corresponding separable classifiers are the same.
\end{remark}

\section{Main Results}
\label{sec:main-results}
In Section~\ref{sec:proposed_test}, we first propose a general stopping rule based on an e-process, constructed using any separable classifier $g$ for a fixed ground-truth distribution tuple $\mathbf{P} \in \mathcal{P}_{\text{sep}}$ (note that this tuple is unknown). This general scheme satisfies the level-$\alpha$ power-one property under $\mathbf{P}$ and $g$; that is, $\mathbb{P}_{\mathbf{P}, 0}(\tau<\infty \mid \mathcal{F}_0) \leq \alpha$ and $\mathbb{P}_{\mathbf{P}, \theta}(\tau<\infty \mid \mathcal{F}_0) = 1$. By conditioning on the initial filtration $\mathcal{F}_0$, we treat the classifier $g$ as fixed, thereby decoupling the design of the stopping rule from the offline training process. Furthermore, we establish an almost sure identification result: as the testing process proceeds, the proposed test identifies the ground-truth index $\theta \in \mathcal{L}$ almost surely. Then, in Section~\ref{sec:role_of_clf}, we investigate the role of the offline training process. Specifically, we study the impact of the training sample size $N$ and the classifier family $\calG$, demonstrating some necessary and sufficient conditions for the size $N$ and how $N$ and $\calG$ affect the stopping time. %

\subsection{The proposed stopping rule and analysis}\label{sec:proposed_test}
Since we fix a tuple $\mathbf{P} \in \mathcal{P}_{\text{sep}}$ and assume a corresponding separable classifier $g$ is given, to simplify the notation and reflect that we have isolated the sequential testing phase from the randomness of the offline training process, we drop $\mathbf{P}$, and the conditioning on $\mathcal{F}_0$ from the probability measures and expectations. Specifically, we denote $\mathbb{P}_{\theta}(\tau<\infty) \equiv \mathbb{P}_{\mathbf{P}, \theta}(\tau<\infty | \mathcal{F}_0)$ and $\mathbb{E}_\theta[\tau]\equiv\mathbb{E}_{\mathbf{P}, \theta}[\tau|\mathcal{F}_0]$ for any $\theta \in \mathcal{L}$. The dependency of the training phase will be studied in Section~\ref{sec:role_of_clf}. Moreover, in this section, for notational simplicity, we define $\bp_\theta[m] \defined \mathbb{P}_{X \sim P_\theta}\lp g(X)=m \rp$ and $\bp_\theta \defined [\bp_{\theta}[0],\bp_{\theta}[1],\ldots,\bp_{\theta}[L]]$, for all $m,\theta\in\calL$.

\subsubsection{Proposed stopping rule}
\label{subsubsec:proposed-stopping-rule}

The design of our proposed stopping rule is built upon the principle of testing-by-betting \cite{Shafer2019}. This says that the amount of statistical evidence against a null is precisely quantified by the gain in wealth made by a fictitious bettor playing a repeated betting game with fair (or sub-fair) odds under the null. Formally, this involves constructing an \emph{e-process}, defined in \cite[Definition 7.3]{Ramdas2025}, such that under the null, the wealth remains low at any time; however, under the alternative, we hope that the amount of the wealth grows (exponentially) fast against the null. Then, the stopping time $\tau\in\mathbb{N}\cup\{+\infty\}$ is defined as the first $1/\alpha$-crossing of the betting-wealth process. In fact, if the betting-wealth process is a non-negative martingale (or supermartingale) under the null, then Ville's inequality immediately guarantees that the probability of stopping in finite time is upper-bounded by $\alpha$, thereby achieving the level-$\alpha$ property.

In words, we construct a betting-wealth process $\{W_n: n \geq 0\}$ adapted to the natural filtration that accumulates the evidence against the null. Hence, the stopping time (or stopping rule) is as follows: 
\[
\tau = \inf \{n \geq 1: W_n \geq 1/\alpha\}.
\]
Then, we turn to the construction of the e-process $\{W_n\}_{n\ge0}$ satisfying the properties mentioned earlier. 

In practice, constructing an e-process typically involves taking the product of sequential $e$-values, modulated by a predictable betting parameter (denoted as $\lambda$) \cite[Chapter 7]{Ramdas2025}. This formulation ensures the process forms a non-negative supermartingale under the null hypothesis while achieving rapid growth under the alternative. Specifically, to address the nonparametric nature of our underlying data distribution, we process the raw sequence $\{X_t\}_{t\ge1}$ using a classifier $g$, projecting the complex observations onto a tractable, discrete label space. We then utilize the resulting sequence of predictions $\{g(X_t)\}_{t\ge1}$ to build our e-process.

The key idea of our construction is that, based on the separability condition of the classifier, if the testing data $X$ is generated by $P_\vartheta$, the probability of $g(X)=\vartheta$ is the largest; we can use this ``edge'' to design betting games accordingly.
Then, the e-process $\{W_n\}_{n\ge0}$ is defined as follows: $\forall\,n\ge0$,
\begin{equation}
        W_0 = 1, \qtext{and} W_n = \int_{-1}^0 \lp \prod_{t=1}^n \lp 1 + \lambda \times \lp \boldsymbol{1}_{g(X_t)=0} - \boldsymbol{1}_{g(X_t)=\hat{j}_t} \rp \rp \rp d\lambda, \label{eq:e-process}
\end{equation}
where $\hat{j}_t$ for every $t \geq 1$ is defined as the index with the largest empirical probability:
\begin{align}
    \hat{j}_t = \argmax_{\theta \in \calL} \; \hat{\bp}_{t-1}[\theta] \label{eq:j-selection}
\end{align}
and $\hat{\bp}_t$ denotes the empirical estimate of the classifier's probability vector, i.e., for each $t\ge1$, $\hat{\bp}_t:=[\hat{\bp}_t[0],\hat{\bp}_t[1],\ldots,\hat{\bp}_t[L]]$, where $\hat{\bp}_t[\theta]:=(1/t)\sum_{s=1}^{t}\mathbf{1}_{g(X_s)=\theta}$ for each $\theta\in\mathcal{L}$, with the initialization $\hat{\bp}_0 = [1/(L+1), \ldots, 1/(L+1)]$, and $\mathbf{1}_{(\cdot)}$ is the indicator function.
We assume that ties in the selection of $\hat{j}_t$ are broken in a deterministic manner (such as selecting the smallest index in the case of ties). 

By the separability condition of the classifier, the process $\{1+\lambda \lp \boldsymbol{1}_{g(X_t)=0} - \boldsymbol{1}_{g(X_t)=\hat{j}_t} \rp\}_{t\ge1}$ forms a sequence of valid e-values under the null, since $ \mathbb{E}_{0}\lb 1 + \lambda\lp \boldsymbol{1}_{g(X_t)=0} - \boldsymbol{1}_{g(X_t)=j} \rp \rb \leq 1$ for all $\lambda \in (-1, 0]$, $j\in\calL$, and $t\ge1$. Therefore, initialized at $W_0=1$, the wealth process $\{W_n\}_{n\ge0}$ is inherently a non-negative supermartingale (and thus an e-process) under the null, immediately establishing the required level-$\alpha$ error control verified by Ville's inequality.
When the alternative hypothesis is true, at a high-level, the sequence of indices $\{\hat{j}_t\}_{t\ge1}$ effectively tracks the most probable class (it holds true under either the null or the alternative), allowing us to accurately identify that $j=\theta$ if $P_\theta$, $\theta\in[L]$, is the ground-truth distribution. %
Then, again by the separability condition of the classifier, we have $\mathbb{E}_\theta\lb 1 + \lambda\lp \boldsymbol{1}_{g(X)=0} - \boldsymbol{1}_{g(X)=j} \rp \rb > 1$ for any $\lambda \in (-1, 0]$. Therefore, the e-process $\{W_n\}_{n\ge0}$ grows as $n$ goes large. These are exactly the properties expected of a good e-process for this testing problem. Finally, instead of selecting a fixed betting parameter $\lambda \in (-1, 0]$, we take the mixture approach, ``averaging'' over all $\lambda\in(-1,0]$. This design choice is a technical necessity for bounding the expected stopping time, allowing us to leverage regret bounds from online convex optimization \cite{hazan2016introduction} to establish our theoretical guarantees. 

To end this section, we provide the following remark on implementing the integration for our proposed e-process.

\begin{remark}[Practical implementation of our e-process]
    In practice, the integral over $\lambda\in(-1,0]$ in the definition of $W_n$ can be approximated by a simple discretization on a sufficiently fine grid, which leads to an efficient \emph{iterative} implementation. With $\{\lambda_i\}_{i=1}^K\subset(-1,0]$ denoting a grid of $K$ points,  define 
\[
M_{n,i}
\;\defined\;
\prod_{t=1}^{n}\Bigl(1+\lambda_i\bigl(\boldsymbol{1}_{g(X_t)=0}-\boldsymbol{1}_{g(X_t)=\hat{j}_t}\bigr)\Bigr),
\qquad M_{0,i}\defined 1, \quad \text{for each } i \in [K].
\]
Since the supermartingale property is preserved under averaging, we can then define the Monte-Carlo process $\bar{W}_n \;\defined\; \frac{1}{K}\sum_{i=1}^{K} M_{n,i}$.
This process can be implemented via a recursive update rule, where upon observing $X_n$ in round $n \geq 1$, we compute $\hat{j}_n$ and then update each component $M_{n,i}$ to obtain $\bar{W}_n$ as follows: 
\begin{align}
&M_{n,i}
= M_{n-1,i}\times\Bigl(1+\lambda_i\bigl(\boldsymbol{1}_{g(X_n)=0}-\boldsymbol{1}_{g(X_n)=\hat{j}_n}\bigr)\Bigr), \qquad \forall\,i\in[K],\qquad\bar{W}_{n}
= \frac{1}{K}\sum_{i=1}^{K} M_{n,i}.
\end{align}
This process can then be used to define the stopping time $\bar{\tau}  \;\defined\;\inf\{n\ge 1:\bar{W}_n \ge 1/\alpha\}$, which in an approximation of the stopping time $\tau$ constructed using the process in~\eqref{eq:e-process}. 
The time complexity of each iteration is $O(K)$.
\end{remark}

\subsubsection{Theoretical analysis of the proposed test}
In this section, we formally describe the theoretical results of our proposed stopping rule. Our first result of this section, Theorem~\ref{theorem:sequential-test}, establishes that the proposed stopping rule $\tau$ defined in~Section~\ref{subsubsec:proposed-stopping-rule}  satisfies the level-$\alpha$ power-one criterion under a given separable classifier $g$ for a given $\mathbf{P}\in\mathcal{P}_{\text{sep}}$. Moreover, an upper bound of the expected stopping time is also provided. Our next result considers the question of identification  of the true alternative index, and in~Proposition~\ref{theorem:identification}, we show that $\hat{j}_t$ equals the true index with probability one after a (random) finite time. 
\begin{theorem}
    \label{theorem:sequential-test} 
     Our proposed stopping rule $\tau$ satisfies the following: under a given separable classifier $g\in\calG$ for a tuple $\mathbf{P}\in\mathcal{P}_{\text{sep}}$,
    \begin{align}
        &\text{Under } H_0: \; \mathbb{P}_{0}\lp \tau < \infty \rp \leq \alpha \\
        &\text{Under } H_1: \; \forall\,\theta\in[L], \; \mathbb{P}_{\theta}\lp \tau < \infty \rp  = 1 \qtext{and} \mathbb{E}_{\theta}\lb \tau \rb = \mathcal{O}\lp \frac{\log(1/(\alpha \Delta_\theta))}{\Delta_\theta^2} \rp, 
    \end{align}
    where $\Delta_{\theta} := \max_{m \in \calL} \bp_\theta[m] - \bp_{\theta}[0] = \bp_\theta[\theta] - \bp_{\theta}[0]$.
\end{theorem}
The proof of Theorem~\ref{theorem:sequential-test} is in Appendix~\ref{proof:sequential-test}. The level-$\alpha$ property follows directly from Ville's inequality, leveraging the fact that $\{W_n\}_{n\ge0}$ is a non-negative supermartingale. Regarding the expected stopping time, while we derive a non-asymptotic upper bound valid for all $\alpha \in (0,1)$, the theorem statement presents only the dominant first-order term.
The lower-order terms, which consist of problem-dependent constants related to $\Delta_\theta$ but are independent of $\alpha$, are deferred to the detailed proof.

In addition to the classic tradeoff between the $\alpha$ value and the expected stopping time,  Theorem~\ref{theorem:sequential-test} also reveals a fundamental dependence on the classifier's ability to separate the null from the alternative labels. This notion of separability is rigorously quantified by the gap in the classifier-induced probabilities $\Delta_\theta$. Consequently, a classifier with higher discriminative power naturally yields a shorter expected stopping time.
Motivated by this dependence, a natural research direction is leveraging heterogeneous classifiers to improve the expected stopping time. In \Cref{sec:exp_mixture}, we show empirically that combining multiple classifiers, e.g., by constructing a mixture of parallel e-processes, can reduce the expected stopping time when compared to test based on a single classifier. This provides an interesting future direction for the design and analysis of aggregation methods that exploit classifiers' heterogeneity.

Finally, the proposed test not only satisfies the level-$\alpha$, power-one criterion, but also enjoys an identification guarantee that says that for any $\theta\in\calL$, $\lim_{t\to\infty}\hat{j}_t=\theta$ almost surely. In other words, as the number of testing samples grows, the procedure identifies the ground-truth index with probability one. 
\begin{proposition}\label{theorem:identification}
   Under a given separable classifier $g$ for a given $\mathbf{P}\in\mathcal{P}_{\text{sep}}$, for any $t\ge1$,
    \[
    \mathbb{P}_{\theta}\lp \hat{j}_t\neq\theta \rp\in O(e^{-t}),
    \]
    implying $\mathbb{P}_\theta\lp \{\hat{j}_t\neq\theta\}\;i.o.\rp=0$ and $\mathbb{P}_{\theta}\lp \exists T<\infty\,s.t.\,\forall\,t\ge T,\,\hat{j}_t=\theta \rp=1$.
\end{proposition}
This statement, proved in Appendix~\ref{proof:identification} using Hoeffding's inequality and the Borel-Cantelli Lemma, then immediately implies the following corollary.
\begin{corollary}\label{cor:indication}
    For any $\theta\neq0$, we have $\lim_{\alpha\rightarrow0}\mathbb{P}_\theta(\hat{j}_\tau=\theta)=1$.
\end{corollary}
The proof is also provided in Appendix~\ref{proof:identification}.
Intuitively, as $\alpha\to 0$ (equivalently, as the stopping threshold $1/\alpha\rightarrow\infty$), the test stops later and the probability of correctly identifying the ground-truth index at stopping approaches one by Proposition~\ref{theorem:identification}. 

\begin{remark}[Asymptotic optimality]
    Generally, showing that our derived expected stopping time is asymptotically optimal, for a fixed classifier $g$ and tuple $\mathbf{P}$ as $\alpha \to 0$, remains an open question. However, there exists a problem instance such that the expected stopping time bounded in Theorem~\ref{theorem:sequential-test} is asymptotically optimal. We defer the formal demonstration of this result to Section~\ref{sec:extension_scd}, where we establish this optimality as a consequence of our framework's application to sequential change-point detection.
\end{remark}
\begin{remark}[Training-testing mismatch]
    The statement of Theorem~\ref{theorem:sequential-test} relies only on the separability condition, and thus its conclusions remain valid under some types of mismatch between training and testing distributions. Specifically, suppose the testing sample under label $\theta$ is generated from a mismatched distribution $\tilde P_\theta$ instead of $P_\theta$ for all $\theta\in\calL$. If $\tilde{\mathbf{P}}:=\{\tilde P_0,\ldots,\tilde P_L\}\in\mathcal{P}_{\text{sep}}$ with a classifier $g$ trained by the exact ground-truth tuple $\mathbf{P}$ (in words, $g$ is also a separable classifier for $\tilde{\mathbf{P}}$), then Theorem~\ref{theorem:sequential-test} also continues to hold with an appropriately modified gap parameter. In particular, the quantity $\Delta_\theta$ in the theorem is replaced by $\tilde{\Delta}_\theta \defined \tilde{\bp}_\theta[\theta]-\tilde{\bp}_\theta[0]$ for all $\theta\in\calL$, where $\tilde{\bp}_\theta[m]:=\mathbb{P}_{X\sim\tilde P_\theta}(g(X)=m)$. The intuition is that in our analysis, we only consider the randomness of the online testing sequences $\{g(X_t)\}_{t\ge1}$ where $g$ is given.
We illustrate this generalization in detail in Section~\ref{sec:mismatch}.
\end{remark}
\begin{remark}[Application to the sequential change-point detection problem]
    Our proposed stopping rule can also be extended to the \emph{sequential change-point detection problem}. 
Specifically, we regard $P_0$ as the \emph{pre-change} distribution and $P_\theta$, for $\theta\in[L]$, as the \emph{post-change} distribution. 
To accommodate an unknown change time, we make a minor modification to our proposed test by recording a set of candidate e-processes, each initialized at a different time index, and monitoring their evidence over time. 
The details are provided in Section~\ref{sec:extension_scd}.
\end{remark}

\subsection{Role of the offline training phase}\label{sec:role_of_clf}
In the previous section, we fixed a classifier and studied the proposed stopping rule constructed via the e-process $\{W_n\}_{n\ge0}$. By doing so, we ignore the training phase based on the offline dataset, i.e., decoupling the stopping rule from the training sample size $N$ and the classifier family $\calG$ (throughout this paper, we only focus on these two variables in the training phase). However, it remains unclear how large $N$ is needed (or must be) to guarantee achievability, and how the interplay between $N$, $\calG$, and the performance of the overall test. Consequently, two fundamental questions remain: (i) For any $\mathbf{P}\in\mathcal{P}_{\text{sep}}$, is it theoretically possible to empirically obtain a corresponding separable classifier, and what is the minimum required training sample size $N$ for a classifier-based level-$\alpha$ power-one test? (ii) How do $N$ and $\calG$ impact the stopping time? We study Question (i) in Section~\ref{sec:veri_separability} and Question (ii) in Section~\ref{sec:minimax}.

\subsubsection{Empirical attainability and training sample requirements} \label{sec:veri_separability} 
In this section, we first establish a sufficient condition on the training sample size $N$ to guarantee that a separable classifier can be empirically obtained via an ERM argument. Subsequently, we derive a necessary condition on $N$ for any overall test $\pi$ achieving the level-$\alpha$ power-one criterion under a given $\mathbf{P} \in \mathcal{P}_{\text{sep}}$. Together, these two results delineate the fundamental limits of $N$, moving beyond the trivial intuition that ``more data is always better'' to mathematically quantify the strict data requirements for achievability. We make a remark for our above argument.
\begin{remark}
    We focus on deriving the lower bound of $N$, and leave obtaining tight upper bound on the training sample size $N$ to future work, as such bounds will depend heavily on the specifics of the learning policy $\phi$, which is not the focus of  this paper. Moreover, our ERM analysis serves mainly as a theoretical proof of concept to establish the \emph{existence} of a valid empirical method. In practice, as demonstrated in our experiments (Section~\ref{sec:experiments}), a separable classifier can be readily obtained using standard training procedures for Feedforward Neural Networks (FNNs) and Convolutional Neural Networks (CNNs).
\end{remark}
We now  proceed to the derivation of a sufficient condition on $N$ for obtaining a separable classifier under any given $\mathbf{P}\in\mathcal{P}_{\text{sep}}$ through an ERM method.
\paragraph{Empirical attainability}
Recall that under any $\mathbf{P} \in \mathcal{P}_{\text{sep}}$, we are given offline training data $\{T_\vartheta^N\}_{\vartheta\in\calL}$, where for each $\vartheta\in\calL$, $T_\vartheta^N \defined (T_{\vartheta,1},\ldots,T_{\vartheta,N}) \sim P_\theta^{\otimes N}$.
For any classifier $g\in\mathcal{G}$ and any $\vartheta,m\in\calL$, define the empirical and confusion probabilities
\[
\hat{p}_{\vartheta,m}(g)
\defined
\frac{1}{N}\sum_{i=1}^{N}\mathbbm{1}\{g(T_{\vartheta,i})=m\},
\qquad
p_{\vartheta,m}(g)
\defined
\mathbb{P}_{X\sim P_\vartheta}\bigl(g(X)=m\bigr).
\]
Note that $p_{\vartheta,m}(g)$ coincides with $\bp_\vartheta[m]$ used in Section~\ref{sec:proposed_test}, but we keep the dependence on $g$ explicit here for clearly constructing the ERM-argument.
Define the empirical gap $\hat{\Delta}(g)$ for any $g\in\calG$ and a data-driven classifier $\hat{g}$:
\[
\hat{\Delta}(g)
\defined
\min_{\vartheta\in\calL}\min_{m\in\calL\setminus\{\vartheta\}}\Bigl(\hat{p}_{\vartheta,\vartheta}(g)-\hat{p}_{\vartheta,m}(g)\Bigr),
\qquad
\hat g \in \argmax_{g\in\mathcal{G}}\hat{\Delta}(g).
\]
With this background, we now state the requirements on $N$ to ensure that $\hat{g}$ is separable. 
\begin{proposition}[Empirical attainability (\textbf{Informal})]
\label{prop:veri_assumption}
Suppose that for any $\mathbf{P} \in \mathcal{P}_{\text{sep}}$, there exists $g^*\in\mathcal{G}$ such that $\Delta^*
\defined
\min_{\vartheta\in\calL}\min_{m\in\calL\setminus\{\vartheta\}}\Bigl(p_{\vartheta,\vartheta}(g^*)-p_{\vartheta,m}(g^*)\Bigr)
>0$ by the definition of $\mathcal{P}_{\text{sep}}$.
Let $\gamma\defined \Delta^*/8$, and for any $\delta\in(0,1)$,
suppose that $N \gtrsim \frac{1}{\gamma^2} \log\left(\frac{L}{\delta}\right)$,
where $d$ is a positive constant related to the VC-dimension of the classifier family $\calG$ and is independent of $\delta$, $L$, $\gamma$, and $N$.
Then,  with probability at least $1-\delta$, the ERM solution $\hat g$ satisfies $p_{\vartheta,\vartheta}(\hat g)>\max_{m\in\calL\setminus\{\vartheta\}}p_{\vartheta,m}(\hat g)$ for any $\vartheta\in\calL$. 
\end{proposition}
The proof idea is built on VC theorem \cite[Theorem 22.18]{wasserman2004all}, and the details of the argument are in Appendix~\ref{proof:assumption}. We note that this lower bound fundamentally depends on the underlying separability margin, $\Delta^*$, achieved by the (optimal) classifier $g^* \in \calG$. Specifically, a smaller $\Delta^*$ indicates that the ground-truth distributions are more difficult for $g^*$ to classify correctly. %
This formalizes the intuition that if the distributions within $\mathbf{P} \in \mathcal{P}_{\text{sep}}$ are inherently ``closer'' or ``less distinguishable'', more offline training data is required to correctly separate them.

After introducing the sufficient condition of $N$ to achieve the empirical attainability, we now investigate a necessary condition on the training sample size: the minimum required number of training samples for all classifier-based sequential tests achieving the level-$\alpha$ power-one property.
\paragraph{Necessary training sample requirement}
We start with the observation that if a classifier-based sequential test $\pi=(\phi, \tau)$ satisfies the level-$\alpha$ power-one criterion for a given tuple $\mathbf{P} \in \mathcal{P}_{\text{sep}}$, it inherently satisfies this criterion for any permutation of $\mathbf{P}$, as defined in Remark~\ref{remark:permutation}. In other words, merely relabeling the hypothesis and the distribution pairs does not compromise the test's validity. We regard it as a natural assumption; by the definition of $\mathcal{P}_{\text{sep}}$, both the original and permuted distribution tuples share the same separable classifier $g \in \calG$. 
Building upon this assumption, we can obtain the following result. 
\begin{proposition}[Lower bound of $N$]\label{prop:lower_bound_of_N}
    For any $\mathbf{P}\in\mathcal{P}_{\text{sep}}$, suppose a policy $\pi=(\phi,\tau)$ that achieves level-$\alpha$ power-one criterion under $\mathbf{P}$ and satisfies the assumption mentioned above, then the number of training samples for $\pi$ must satisfy 
    \[
    N\geq\frac{\log(1/\alpha)}{\min_{\theta\neq m}\mathrm{J}(P_m, P_\theta)}, \qtext{where} \mathrm{J}(P,Q):=D_{\text{KL}}(P\Vert Q)+D_{\text{KL}}(Q\Vert P). 
    \]
\end{proposition}

The proof of Proposition~\ref{prop:lower_bound_of_N} is provided in Appendix~\ref{proof:lower_bound_of_N}, and it relies mainly on an application of the data processing inequality \cite[Chapter 3]{Polyanskiy_Wu_2025} for relative entropy. 
This proposition shows that when two ground-truth distributions are nearly indistinguishable, more offline samples are needed to learn a classifier that reliably separates all classes. We take the minimum over all $P_m$ and $P_\theta$ since we assume that the test $\pi$ satisfies the level-$\alpha$ power-one criterion for all permutation tuples. Hence, the ground-truth distributions of the hypotheses may change (relabeling) while the test's validity is maintained.

\subsubsection{Minimax lower bound}\label{sec:minimax}
To directly answer how the training sample size $N$ and the classifier family $\mathcal{G}$ impact the stopping time, we now establish  a minimax lower bound on an appropriate notion of ``type-II error''.
However, by shifting our focus to the minimax lower bound on the type-II error, we leverage Le Cam's two-point method \cite[Chapter 31]{Polyanskiy_Wu_2025}, an information-theoretic technique for the minimax risk's lower bound, to derive a bound that explicitly characterizes the joint impact of the training sample size and the classifier family.

To present our minimax lower bound, we need specify the test class and the distribution set over which the infimum and supremum are evaluated. For the distribution set, we consider $\mathcal{P}_{\text{sep}}$ introduced in~\eqref{eq:P-sep}. For the test class, we define the \emph{uniformly achievable} test class, denoted by $\Pi_\alpha(\mathcal{P}_{\text{sep}})$, as  all  tests $\pi$ that satisfy the level-$\alpha$ and power-one criteria uniformly over the entire family $\mathcal{P}_{\text{sep}}$. That is, 
\begin{align} \label{eq:policy_class}
    \Pi_\alpha(\mathcal{P}_{\text{sep}}) \coloneqq \left\{ \pi : \sup_{\mathbf{P} \in \mathcal{P}_{\text{sep}}} \mathbb{P}^{\pi}_{\mathbf{P}, 0}(\tau<\infty) \leq \alpha \qtext{and} \inf_{\mathbf{P} \in \mathcal{P}_{\text{sep}}} \mathbb{P}^{\pi}_{\mathbf{P}, \theta}(\tau<\infty) = 1 \ \ \forall \theta \in [L] \right\},
\end{align}
where $\mathbb{P}^{\pi}_{\mathbf{P},\theta}(\cdot)$ denotes the probability measure defined in Section~\ref{sec:problem_form}, but we add $\pi$ as the superscript to emphasize that the classifier-training ($\phi$) and the stopping rule ($\tau$) are decided by the overall policy $\pi=(\phi,\tau)$.
Moreover, we note that the stopping time inherently depends on the trained classifier, which is derived from training samples drawn from the underlying distribution $\mathbf{P}$. Consequently, evaluating the supremum and infimum in the set $\Pi_{\alpha}(\mathcal{P}_{\text{sep}})$ implies that for any $\mathbf{P}$, the learning policy $\phi$ trains a classifier $g \in \calG$ on data samples generated from that specific $\mathbf{P}$. This resulting classifier is then utilized to construct the stopping rule. We remark that if for any $\mathbf{P}\in\mathcal{P}_{\text{sep}}$, one can find the corresponding separable classifier (we theoretically show that it can be found with high probability in Section~\ref{sec:veri_separability}), then combining it with our proposed stopping rule, the induced overall test is in $\Pi_{\alpha}(\mathcal{P}_{\text{sep}})$.  With these two terms, we can define the minimax error as 
\begin{align}
    \Psi \equiv \Psi(\Pi_\alpha, \calP_{\text{sep}}, \calG) \coloneqq \inf_{\pi\in\Pi_{\alpha}(\mathcal{P}_{\text{sep}})}\sup_{\mathbf{P}\in\mathcal{\mathcal{P}_{\text{sep}}}}\mathbb{P}^\pi_{\mathbf{P},\theta}(\tau > n). \label{eq:minimax-error}
\end{align}

To obtain a lower bound on $\Psi$, we will use Le Cam's two-point method, by working with two parameterized underlying tuples $\mathbf{P}_{\nu_0}:=\{P_{i,\nu_0}\}_{i\in\calL}$ and $\mathbf{P}_{\nu_1}:=\{P_{i,\nu_1}\}_{i\in\calL}$, where $P_{i,\nu_m}$ is parameterized by $\nu_m$ for $m\in\{0,1\}$ for each $i\in\calL$ by $\nu_0$ and $\nu_1 = \nu_0 + \delta$ (for some $\delta > 0$), representing two closely related but distinct environments (we assume that $\mathbf{P}_{\nu_0}$ and $\mathbf{P}_{\nu_1}$ both are in $\mathcal{P}_{\text{sep}}$). %
Furthermore, we capture the capacity limitation of the chosen classifier family $\mathcal{G}$ through the following structural assumption.
\begin{assumption}[Classifier capacity limit (\textbf{Informal})]\label{assum:classifier_capacity_limit}
    No single classifier $g \in \mathcal{G}$ can be simultaneously optimal (in terms of distinguishability) for both environments $\nu_0$ and $\nu_1$. Specifically, the minimum aggregate suboptimality gap (we will define in Appendix~\ref{proof:type-II_lower_bound}) of any $g \in \mathcal{G}$ across both environments is strictly lower bounded by a constant $B > 0$. The formal statement of this assumption is provided in Assumption~\ref{assump:classifier_capacity_limit_formal}.
\end{assumption}
This constant $B$ represents the restriction imposed by $\mathcal{G}$. A more restricted family $\mathcal{G}$ typically yields a larger $B$, as it struggles  to adapt to both environments simultaneously more than a richer family. With this conceptual setup, we present the minimax lower bound, which demonstrates the impact of $N$ and the classifier family capacity $B$ (representing the family $\calG$).

\begin{theorem}[Minimax lower bound]\label{thm:type-II_lower_bound}
Suppose $\nu_0$ and $\nu_1$ satisfy that $\nu_1=\nu_0+\delta$ for some $\delta>0$ and there exists a finite constant $M>0$ such that $D(P_{i,\nu_0}\Vert P_{i,\nu_1})\leq M \delta^2$ for any $i\in\calL$.
Then, under Assumption~\ref{assum:classifier_capacity_limit}, 
\begin{align}
&\log\Psi\defined\log\inf_{\pi\in\Pi_{\alpha}(\mathcal{P}_{\text{sep}})}\sup_{\mathbf{P}\in\mathcal{\mathcal{P}_{\text{sep}}}}\mathbb{P}^\pi_{\mathbf{P},\theta}(\tau > n)
\ge 
-\frac{1}{1-\alpha}\!\left(
n\Big(
\max_{\nu\in\{\nu_0,\nu_1\}} D(P_{0,\nu}\Vert P_{\theta,\nu})
-\frac{B}{4}e^{-NM(L+1)\delta^2
}\Big)
+\log 2\right)
\end{align}
for any alternative index $\theta\in\calL\setminus\{0\}$.
\end{theorem}
The proof is deferred to Appendix~\ref{proof:type-II_lower_bound}. %
We provide the following interpretation of some parameters in this lower bound: 
\textbf{1) Tradeoff between type-I and type-II error}: As $\alpha$ increases, the level-$\alpha$ constraint becomes looser and the sequential test is allowed to stop more easily. Therefore, the minimax lower bound decreases. Conversely, a smaller $\alpha$ means a stricter level constraint, which increases the lower bound. 
\textbf{2) Problem instance matters:} The KL-divergence in the right-hand side reflects the intrinsic difficulty of distinguishing distributions of the worst-case environment over the two chosen ones. 
\textbf{3) The classifier family matters:} If the classifier $\calG$ is more restricted ($B$ is larger), which in turn increases the minimax lower bound. 
\textbf{4) Training sample size matters:} When $N$ is large, the offline data provide a more accurate representation of the underlying distributions, so the classifier can learn better and the lower bound decreases. Notably, even as $N\to\infty$, the bound does not vanish: the stopping time is still dominated by the intrinsic testing difficulty represented by the supremum term.

\begin{remark}[Minimax lower bound for the testing-training mismatch setting]
In Section~\ref{sec:mismatch}, we study a setting where the testing distribution exhibits a shift from the training distribution. A corresponding minimax lower bound for this mismatched scenario can be readily established by following the proof of Theorem~\ref{thm:type-II_lower_bound}. However, the bound is not clean and provides less insight. We leave it as future work.
\end{remark}

\section{Extensions}\label{sec:extension}
In Section~\ref{sec:proposed_test}, we have introduced a general stopping rule that satisfies the level-$\alpha$ power-one criterion and analyzed its expected stopping time. Subsequently, we investigated the fundamental roles of the training sample size and the classifier family. In this section, we extend our proposed stopping rule to accommodate one practical scenario: training-testing distribution mismatch, along with one concrete application: sequential change point detection. We note that analysis regarding the training sample size and the classifier family is omitted in these extensions, as such theoretical guarantees are either similar to our previous discussions or out of the scope of this paper. We will focus on the extensions themselves.

\subsection{Training-testing mismatch}\label{sec:mismatch}
In many practical applications \cite{Pan2010}, the data distributions between model training and online testing exhibit a \emph{mismatch}. Specifically, in our framework, suppose the classifier is trained on offline samples $T_\vartheta^N \coloneqq \{T_{\vartheta, 1}, \ldots, T_{\vartheta, N} \} \simiid P_\vartheta$ for $\vartheta \in \calL$, while the online testing samples are drawn i.i.d. from a testing tuple $\tilde{\mathbf{P}} = \{\tilde P_0,\ldots,\tilde P_{L}\}$, where $\tilde P_\theta$ may not be identical to $P_\theta$. 

There are many ways to characterize the distribution mismatch, such as the total variation or the KL-divergence between the training and testing distributions. In this section, we first consider the KL-divergence between the classifier-induced probabilities under the training and testing distributions. Specifically, given a separable classifier $g$ for the training distribution tuple $\mathbf{P}$, we define a set of $\epsilon$-mismatch distribution tuples as follows:
\begin{align}
\calE_\epsilon\equiv\calE_\epsilon(g,\mathbf{P})\defined\left\{\tilde{\mathbf{P}}=\{\tilde P_0,\ldots,\tilde P_{L}\}\mid D(\bp_\theta\|\tilde\bp_\theta)\leq \epsilon\qtext{for all}\theta\in\calL\right\},\label{eq:KL-mismatch-set}
\end{align}
where the definition of each $\tilde\bp_\theta$ is similar to the \emph{pmf} $\bp_\theta$ defined in Section~\ref{sec:proposed_test}, but the data sample is generated by $\tilde P_\theta$, i.e., for any $\theta\in\calL$, $\tilde\bp_\theta:=[\tilde\bp_\theta[0],\ldots,\tilde\bp_\theta[L]]$ and $\tilde \bp_\theta[i]=\mathbb{P}_{X\sim\tilde P_\theta}(g(X)=i)$ for all $i\in\calL$. Intuitively, if the mismatch strength $\epsilon$ is sufficiently small such that the classifier $g$ remains separable under the mismatched tuples, one can immediately utilize Theorem~\ref{theorem:sequential-test}. In this case, for $\tilde{\mathbf{P}}\in\calE_\epsilon$, the level-$\alpha$ power-one criterion is achieved, and the expected stopping time is bounded by $\mathcal{O}\left( \frac{\log(1/(\alpha \tilde \Delta_\theta))}{\tilde \Delta_\theta^2} \right)$, where $\tilde \Delta_\theta \coloneqq \tilde\bp_\theta[\theta]-\tilde\bp_\theta[0]$ for any $\theta\in[L]$. However, for larger values of $\epsilon$, the separability of the classifier $g$ can be violated. %

Therefore, we show that the largest $\epsilon$ such that $g$ remains a separable classifier for all $\tilde{\mathbf{P}}\in\calE_\epsilon$ is characterized by the strength of $g$'s separability (the ``edge'' mentioned in Section~\ref{sec:proposed_test}) on the training tuple $\mathbf{P}$. In other words, when the classifier $g$ has stronger separability on $\mathbf{P}$, the set of mismatched but separable tuples can be larger, which makes our proposed stopping time and analysis more robust. Specifically, we provide a sufficient condition for $\epsilon$ such that given $\{\Delta_{\theta,m}\}_{\theta\in\calL,m\in\calL\setminus\{\theta\}}$, where each $\Delta_{\theta,m}\defined\bp_\theta[\theta]-\bp_\theta[m]$, the classifier trained on the exact $\mathbf{P}$ is also separable for all mismatch tuples in the induced $\calE_\epsilon$. We formally describe this result in Proposition~\ref{prop:robust_result}.
\begin{proposition}[Robustness]\label{prop:robust_result}
    Under a fixed separable classifier $g$ for the training tuple $\mathbf{P}$, if the mismatch strength $\epsilon< \min_{\theta \neq m \in\calL}\frac{1}{2}\Delta_{\theta,m}^2$, then the classifier $g$ is also separable for any $\tilde{\mathbf{P}}\in\calE_\epsilon$, i.e., for any $\tilde{\mathbf{P}}\in\calE_\epsilon$, $\tilde{\bp}_\theta[\theta]>\max_{m\in\calL\setminus\{\theta\}}\tilde{\bp}_\theta[m]$, for all $\theta\in\calL$. Moreover, with such condition for the $\epsilon$, Theorem~\ref{theorem:sequential-test} still holds with the the expected stopping time upper bounded by $\mathcal{O}\left( \frac{\log(1/(\alpha \tilde \Delta_\theta))}{\tilde \Delta_\theta^2} \right)$ and $0<\tilde{\Delta}_\theta\leq\Delta_\theta+\sqrt{2\epsilon}$, where $\tilde \Delta_\theta \coloneqq \tilde\bp_\theta[\theta]-\tilde\bp_\theta[0]$, for any $\theta\in[L]$
\end{proposition}
\begin{proof}
The proof of Proposition~\ref{prop:robust_result} is based on Pinsker's inequality~\cite[Theorem 7.10]{Polyanskiy_Wu_2025}.  In particular, consider any $\tilde{\mathbf{P}}\in\calE_\epsilon$, and note that for any $\theta\in\calL$, we have $D(\bp_\theta\|\tilde\bp_\theta)\geq 2\text{TV}^2(\bp_\theta,\tilde\bp_\theta)$ due to Pinsker's inequality. Then, since $\bp_\theta$ and $\tilde\bp_\theta$ are discrete \emph{pmf}, by the definition of the total variation distance and the triangle inequality, $\text{TV}(\bp_\theta,\tilde\bp_\theta)$ is lower bounded by $\frac{1}{2}|\tilde\Delta_{\theta,m}-\Delta_{\theta,m}|$ for any $m\in\calL\setminus\{\theta\}$, where $\tilde\Delta_{\theta,m}\defined\tilde\bp_{\theta}[\theta]-\tilde\bp_\theta[m]$. Therefore, by the definition of $\epsilon$-mismatch set $\calE_\epsilon$, we can further get $|\tilde\Delta_{\theta,m}-\Delta_{\theta,m}|\leq\sqrt{2\epsilon}$ for any $\theta\in\calL$, $m\in\calL\setminus\{\theta\}$. As a result, by $\Delta_{\theta,m}-\sqrt{2\epsilon}\leq\tilde\Delta_{\theta,m}\leq\Delta_{\theta,m}+\sqrt{2\epsilon}$, if we set $\epsilon<\frac{1}{2}(\min_{\theta\in\calL}\min_{m\in\calL\setminus\{\theta\}}\Delta_{\theta,m})^2$, then the lower bound of $\tilde\Delta_{\theta,m}$, $\Delta_{\theta,m}-\sqrt{2\epsilon}$, is positive for any $\theta\in\calL$ and $m\in\calL\setminus\{\theta\}$. Therefore, the classifier $g$ is also separable for any $\tilde{\mathbf{P}}\in\calE_\epsilon$ and  $\tilde\Delta_\theta\equiv\tilde\Delta_\theta[0]>0$ for any $\theta\in[L]$, and thus~Theorem~\ref{theorem:sequential-test} applies. 
\end{proof}
As the proof of~Proposition~\ref{prop:robust_result} suggests, we can also define the mismatch using  total variation; that is,  the $\epsilon$-mismatch set is  
\[
\calE^{\text{TV}}_\epsilon\equiv\calE^{\text{TV}}_\epsilon(g,\mathbf{P})\defined\left\{\tilde{\mathbf{P}}=\{\tilde P_0,\ldots,\tilde P_{L}\}\mid \text{TV}(\bp_\theta,\tilde\bp_\theta)\leq \epsilon\qtext{for all}\theta\in\calL\right\},
\]
where $\text{TV}(P,Q)=\frac{1}{2}\lVert P-Q\rVert_1$ if $P$ and $Q$ are probability mass functions.
With this modification, we can state a version of Proposition~\ref{prop:robust_result}, but with the condition on $\epsilon$ replaced by $\epsilon<\frac{1}{2}\min_{\theta\in\calL}\min_{m\in\calL\setminus\{\theta\}}\Delta_{\theta,m}$, leading to $\tilde\Delta_\theta\in(0,\Delta_\theta+2\epsilon]$ for any $\theta\in\calL$.

Finally, we end this section with a simple but somewhat counterintuitive observation:  distribution shift does not always degrade the expected stopping time. If the shift renders the mismatched tuple $\tilde{\mathbf{P}}$ more distinguishable, pushing the underlying distributions further apart in the classifier's decision space, it is possible that $\tilde \Delta_\theta > \Delta_\theta$ for some $\theta\in[L]$, thereby accelerating the stopping time.

\subsection{Sequential change detection}\label{sec:extension_scd}
In this section, we demonstrate that our proposed stopping rule can be applied to the standard sequential change-point detection (SCD) task \cite{Lai1995}.
The goal in SCD problems is to monitor a stream of observations, and identify any abrupt changes in distributions (at an unknown time $T$ called the \emph{change point}) with minimal delay, while controlling false alarms within a tolerable limit.
Specifically, we assume that there are a pre-change distribution $P_0$ and a post-change distribution $P_\theta$, for $\theta\in[L]$. At some unknown change point $T \in \mathbb{N} \cup \{\infty\}$, the underlying distribution of the testing samples shifts abruptly from $P_0$ to $P_\theta$. Furthermore, these true data-generating distributions are totally unknown to the decision-maker, beyond the access to offline data. In words, for a fixed $\theta\in[L]$, the testing samples follow $X_t \simiid P_0$ when $t<T$, but change to $X_t \simiid P_\theta$ for $t\ge T$.
We write $\mathbb{E}_\infty[\cdot]$ for ``expectation'' under the no-change regime ($T=\infty$), and $\mathbb{E}_{T, \theta}[\cdot]$ for ``expectation'' under a change to $P_\theta$ at some finite time $T<\infty$.

Our proposed e-process in~\eqref{eq:e-process} can be adapted to the SCD problem as follows: Under a tuple $\mathbf{P}\in\mathcal{P}_{\text{sep}}$, let $\{X_n: n \geq 1\}$ denote the testing samples. For any $n \geq 1$, and $1 \leq k \leq n$, let $\{W^{(k)}_n: n \geq k\}$ denote the e-process ``starting from $k$'': 
\[
W^{(k)}_n \defined \int_{-1}^0 \lp \prod_{t=k}^n \lp 1 + \lambda \times \lp \boldsymbol{1}_{g(X_t)=0} - \boldsymbol{1}_{g(X_t)=\hat{j}_t} \rp \rp \rp d\lambda.
\]
Then, we define the $e$-detector \cite{shin2022detectors} process $\{M_n: n \geq 1\}$ as $M_n = \max_{1\leq k\leq n} W_n^{(k)}$ for any $n\ge1$, and build the change detection stopping time: $ \tau = \inf \{n \geq 1: M_n \geq 1/\alpha\}$. 

Thus the SCD procedure initializes a new power-one test in every round, and declares a detection as soon of any of the ``active'' tests reject the null. 
Intuitively, if the change occurs at time $T$ and $n\ge T$, then the process started at $k=T$, namely $W^{(T)}_n$, (and those started after $T$) evolves entirely on post-change samples and should therefore have the same growth behavior as in the sequential hypothesis testing setting under $H_1$. Since the observations are i.i.d.\ within each regime, the resulting expected detection delay should be the same as that in Section~\ref{sec:main-results}. We formalize this in our next result. 
\begin{theorem}
    \label{theorem:change-detection}
    For any distribution tuple $\mathbf{P} \in \mathcal{P}_{\text{sep}}$, let $P_0$ and $P_\theta$, $\theta\in[L]$  denote the pre- and post-change distributions respectively. Then, for any $\theta\in[L]$, we have 
    \begin{align}
        &\mathbb{E}_{\infty}[\tau] \geq 1/\alpha, \qtext{and} 
         \sup_{T \in \mathbb{N}}\esssup \mathbb{E}_{T, \theta}[(\tau-T)^+|\mathcal{F}^{(0)}_T] = \mc{O}\lp  \frac{\log(1/\alpha \Delta_\theta)}{\Delta^2_\theta}\rp, 
    \end{align}
    where $\mathcal{F}^{(0)}_T:=\sigma\{\{g(X_t)\}_{t\leq T}\}$ and $\Delta_\theta$ has the same definition as in~\Cref{theorem:sequential-test}. 
\end{theorem}
The proof of Theorem~\ref{theorem:change-detection} is in Appendix~\ref{proof:change-detection}.
Note that in the SCD setting, the expected detection delay generally depends on the ground-truth post-change distribution $P_\theta$. Moreover, when $\alpha$ is sufficiently small, Proposition~\ref{theorem:identification} implies that, at stopping, $\hat{j}_\tau$ identifies the true post-change index with high probability. This identification is practically important since knowing the post-change parameter can help localize the cause of the anomaly.

We now argue that there exist problem instances for which the expected stopping time derived in Theorem~\ref{theorem:change-detection} cannot be improved. To do so, let us recall the lower bound of the Lorden's criterion~\cite[Section IV-B]{veeravalli2014quickest} applied to the classifier's output sequence instead of the raw data: if the post-changed distribution is $P_\theta$ for $\theta\in[L]$, then
\begin{align}
    \sup_{T\in\mathbb{N}}\esssup\mathbb{E}_{T,\theta}[(\tau-T)^{+}|\mathcal{F}^{(0)}_{T}]\in\Omega\lp\frac{\log(1/\alpha)}{D(\bp_\theta\lVert \bp_0)}\rp\quad\text{as }\alpha\rightarrow 0,
\end{align}
where $\bp_\theta=[\bp_\theta[0],\ldots,\bp_\theta[L]]$ and $\bp_\theta[m]:=\mathbb{P}_{X\sim P_\theta}(g(X)=m)$ for all $m,\theta\in\calL$. We prove that there exists a problem instance and a corresponding separable classifier such that the upper bound of the expected delay in Theorem~\ref{theorem:change-detection} cannot be improved asymptotically as $\alpha\rightarrow\infty$. It is sufficient to show that there exists $\{\bp_\nu\}_{\nu\in\calL}$ satisfying $\bp_{\theta}[\theta]>\max_{m\in\calL\setminus\{\theta\}}\bp_{\theta}[m]$ for all $\theta\in\calL$ such that $D(\bp_\theta\lVert \bp_0)\leq\Delta_\theta^2$.
\begin{proposition}[Optimality in some cases]
Let $L=1$ and $\bp_{0}=[1-e^{-1},e^{-1}]$ and $\bp_1=[0.1, 0.9]$, and let $\mathbb{E}_{T,\theta}[\cdot]$ denote the expectation when there is a change from $\bp_0$ to $\bp_1$ at time $T \in \mathbb{N}\cup \{\infty\}$. Then, we have 
    \begin{equation}
        \sup_{T\in\mathbb{N}}\esssup\mathbb{E}_{T,\theta}[(\tau-T)^{+}|\mathcal{F}_{T}]\in\Omega \lp \frac{\log1/\alpha}{\Delta^2_1} \rp\quad\text{as }\alpha\rightarrow0.
    \end{equation}
\end{proposition}
This result, which is a direct consequence of the fact that relative entropy is bounded by the chi-squared divergence,  implies that there exist problem instances for which  the expected stopping time bound for our proposed stopping rule derived in  Theorem~\ref{theorem:sequential-test} is also asymptotically optimal.

\section{Experiments}
\label{sec:experiments}

After deriving the theoretical results in Section~\ref{sec:main-results}~and~\ref{sec:extension}, we now turn to empirical experiments utilizing both synthetic and real-world data to support our results. In particular, we present the following experiments:
\begin{enumerate}
    \item In Section~\ref{sec:verify_thm_sht}, Theorem~\ref{theorem:sequential-test} is verified by some powerful classifiers, including a simple Feedforward Neural Network (FNN) and a common VGG16~\cite{Simonyan2015}. We consider synthetic data generated by Gaussian distributions and real-world data obtained from CIFAR-10. Moreover, we demonstrate that our proposed method can identify the true data distribution at stopping, verifying Proposition~\ref{theorem:identification}.

    \item In Section~\ref{sec:exp_dist_shift}, we turn to verifying the robustness described in Proposition~\ref{prop:robust_result}, showing that given a separable classifier for a training tuple (i.e., given (empirical) $\{\Delta_{\theta,m}\}_{\theta\in\calL,m\in\calL\setminus\{\theta\}}$), we can define the mismatch strength $\epsilon$ according to Proposition~\ref{prop:robust_result} and demonstrate the empirical average stopping time matches the trend of the theoretical result. %
    
    \item In Section~\ref{sec:verify_thm_scd}, we apply our proposed  change-point detection scheme on synthetic data streams, and verify that the empirical average delay matches the theoretical result. 
    
    \item  In Section~\ref{sec:exp_mixture}, we illustrate how our framework can easily aggregate information from several classifiers. In particular, we  consider a setting with two classifiers and show that taking an average mixture of their betting-wealth processes leads to a shorter average stopping time than using only a single classifier. This inspires one to study adaptive weighting schemes.
\end{enumerate}
\begin{remark}
    Our empirical evaluation focuses specifically on verifying the proposed stopping rule, especially the expected stopping time. Therefore, for training the classifier, we provide only some parameters used in neural networks to get the separable classifier under the specified ground-truth distribution tuple. Also, we do not simulate the level-$\alpha$ and power-one properties since they have been shown in many \emph{testing-by-betting} works, e.g., \cite{Shekhar2024}, that the test martingale indeed can achieve those properties empirically.
\end{remark}

\subsection{Verification of Theorem~\ref{theorem:sequential-test} and Proposition~\ref{theorem:identification}}\label{sec:verify_thm_sht}
In this experiment, we focus on the scenario when the alternative hypothesis ($H_1$) is true and $L=2$. We demonstrate that the \emph{trend} of the empirically expected stopping time is as shown in Theorem~\ref{theorem:sequential-test}. Moreover, we show that the indicator $\hat{j}_t$ in our proposed test can identify the true alternative distribution index. 
\paragraph{Environments:} We consider the following two types of data to provide the above results: 
\begin{itemize}
    \item \textbf{Case 1: (Gaussian distribution)} A synthetic Gaussian dataset paired with a lightweight fully connected neural network classifier (with two hidden layers of sizes 64 and 32), implemented using the function \texttt{MLPClassifier} (Multi-layer perception classifier) in Python~\cite{Buitinck2013}. The detailed hyperparameter settings are provided in our \href{https://github.com/chiayuh-umich/ClassifierBasedSHT}{GitHub} (see \texttt{expt\_NormalModel.py}).

    \item \textbf{Case 2: (CIFAR-10)} A subset of CIFAR-10 images~\cite{Krizhevsky2009} paired with a convolutional neural network classifier, implemented using VGG16~\cite{Simonyan2015}. The detailed architecture and training settings are provided in our GitHub repository (see \texttt{expt\_CIFAR10.py}).
\end{itemize}

Then, for these two cases, we illustrate their data-generating distribution tuples and the corresponding separable classifiers. Note that there is no training-testing mismatch in this section.

For Case 1, we use $10$-dimensional Gaussian vectors with a common covariance matrix to define $\mathbf{P}$: $P_0=\mathcal{N}(\mu_0,\Sigma)$, $P_1=\mathcal{N}(\mu_1,\Sigma)$, and $P_2=\mathcal{N}(\mu_2,\Sigma)$,
where $\Sigma$ is the $10\times 10$ identity matrix and we set 
\begin{align}
    &\mu_0 = [0,\, 0,\, 0,\, 0,\, 0,\, 0.9,\, 0.8,\, 0.9,\, 0.8,\, 0.9].\\
    &\mu_1 = [0,\, 0,\, 0.1,\, 0.1,\, 0.1,\, 0.7,\, 0.8,\, 0.9,\, 0.9,\, 0.9],\\
    &\mu_2 = [0,\, 0,\, 0,\, 0,\, 0,\, 1,\, 1,\, 1,\, 1,\, 1].
\end{align}
The corresponding separable classifier's confusion matrix is shown as Table~\ref{tab:exp1_normal}. For Case 2, we use the first three classes of CIFAR-10, namely $P_0$ (Airplane), $P_1$ (Automobile), and $P_2$ (Bird). The confusion matrix of the separable classifier is described in Table~\ref{tab:exp1_cifar}. Note that since we can find a corresponding separable classifier for those considered ground-truth tuples, those tuples are in $\mathcal{P}_{\text{sep}}$.
\begin{table}[H]
\centering
\caption{Confusion matrices (in probability) of classifiers. Each entry of the table is a numerical version of $\mathbb{P}_{X\sim P_\theta}(g(X)=m)$, for any $\theta,m\in\calL$, evaluated by the sampling data.}
\label{tab:two_cm}
\begin{subtable}[t]{0.48\linewidth}
\centering
\caption{Case 1: The MLP Classifier is trained over $10000$ samples uniformly drawn from the ground-truth tuple ($N\cong10000/3$).}
\label{tab:exp1_normal}
\renewcommand{\arraystretch}{1.15}
\begin{tabular}{c ccc}
\toprule
 & $g(X)=0$ & $g(X)=1$ & $g(X)=2$ \\
\midrule
$\theta=0$ & 0.483609 & 0.243609 & 0.272782 \\
$\theta=1$ & 0.186343 & 0.559332 & 0.254325 \\
$\theta=2$ & 0.200000 & 0.244970 & 0.555030 \\
\bottomrule
\end{tabular}
\end{subtable}
\hfill
\begin{subtable}[t]{0.48\linewidth}
\centering
\caption{Case 2: The VGG16 neural network is trained over $20000$ samples uniformly drawn from the ground-truth tuple ($N\cong20000/3$).}
\label{tab:exp1_cifar}
\renewcommand{\arraystretch}{1.15}
\begin{tabular}{c ccc}
\toprule
 & $g(X)=0$ & $g(X)=1$ & $g(X)=2$ \\
\midrule
$\theta=0$ & 0.867 & 0.065 & 0.068\\
$\theta=1$ & 0.062 & 0.903 & 0.035 \\
$\theta=2$ & 0.14 & 0.048 & 0.812 \\
\bottomrule
\end{tabular}
\end{subtable}
\end{table}

\paragraph{Average stopping time:}
Then, we simulate the average stopping time of our proposed stopping rule and verify the expected stopping time in Theorem~\ref{theorem:sequential-test}. We assume that the alternative $\theta=2$ in both cases. The corresponding results are shown in Figure~\ref{fig:exp1}. 
Consequently, the scaling trend (orange dotted curve) predicted by our theory matches the empirical results well in both Case~1 (Figure~\ref{fig:exp1_normal}) and Case~2 (Figure~\ref{fig:exp1_cifar}).

\begin{figure}[t]
    \centering
    \begin{subfigure}[t]{0.48\linewidth}
        \centering
        \includegraphics[width=\linewidth]{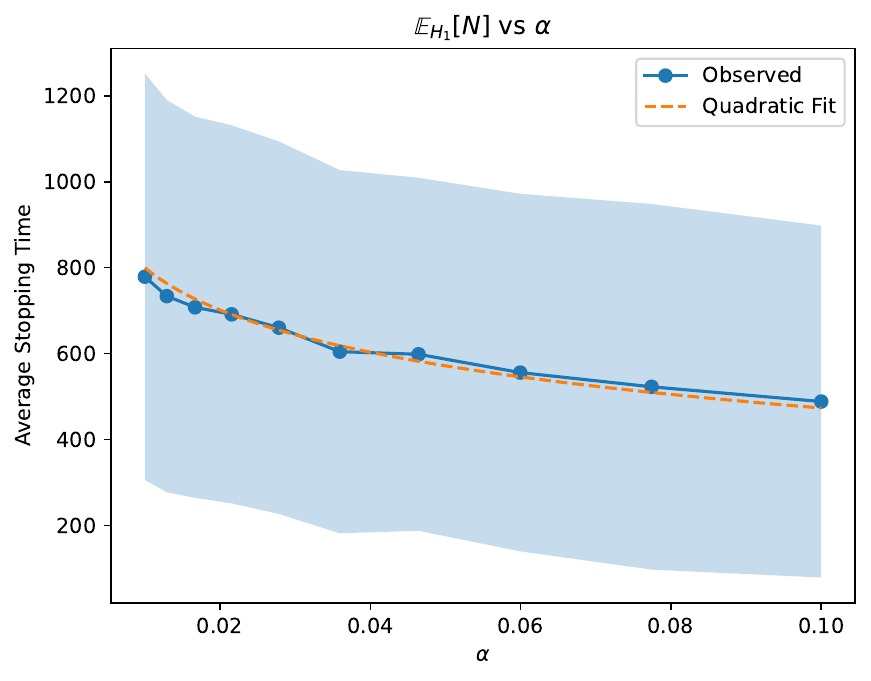}
        \caption{Case 1 with $\Delta_2\cong0.355$}
        \label{fig:exp1_normal}
    \end{subfigure}
    \hfill
    \begin{subfigure}[t]{0.48\linewidth}
        \centering
        \includegraphics[width=\linewidth]{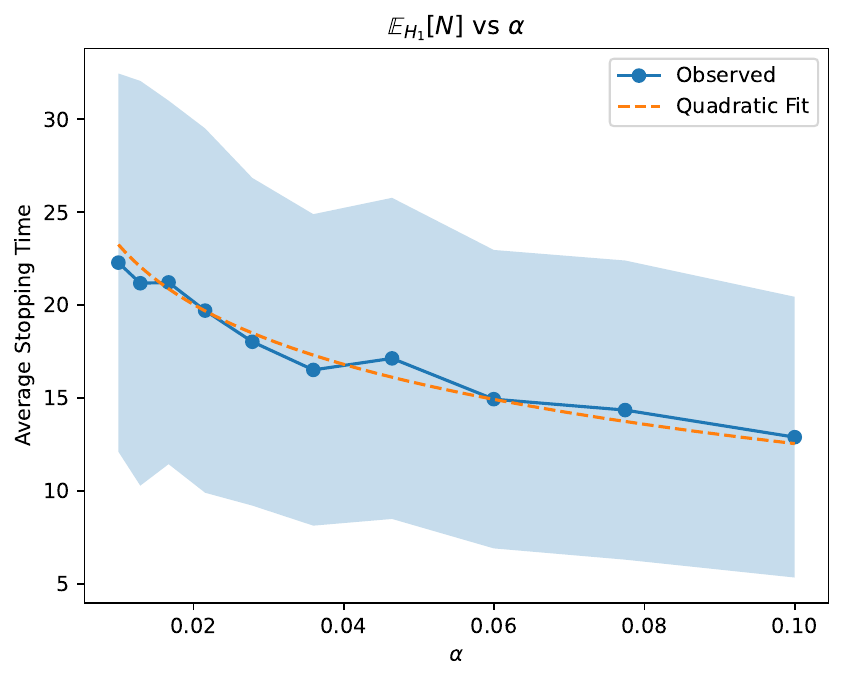}
        \caption{Case 2 with $\Delta_2=0.672$}
        \label{fig:exp1_cifar}
    \end{subfigure}
    \caption{Average expected stopping time vs. desired level-$\alpha$ values under $\theta=2$: In both experiments, we consider $10$ $\alpha$-levels, denoted by $\alpha[1],\ldots,\alpha[10]$. For each $\alpha[i]$, we run $300$ independent trials, where each trial corresponds to one run of the sequential test at level $\alpha[i]$. The observed average stopping time (over $300$ trials) of each $\alpha[i]$ is represented by the blue dots. The curve labeled as ``Quadratic Fit'' is the best match for the empirical results through scaling the constant term of the bound of the theoretical expected stopping time. Details of the quadratic fit is described in Remark~\ref{remark:quadratic_fit}.}
    \label{fig:exp1}
\end{figure}

\begin{remark}[Quadratic fit]\label{remark:quadratic_fit}
For the theoretical curve, we focus on the \emph{scaling trend} of the expected stopping time implied by Theorem~\ref{theorem:sequential-test}, which is named as ``Quadratic Fit'' in this paper. To facilitate a trend-level comparison with the empirical averages, we introduce a scaling constant
\[
c \;:=\; \mathrm{Avg}_{i\in[10]}\left\{\frac{\mathrm{avg\_stopping\_time}[i]\cdot\Delta^2_\theta}{\log\!\bigl(1/(\alpha[i]\cdot\Delta_\theta)\bigr)}\right\},
\]
where $\mathrm{avg\_stopping\_time}[i]$ denotes the empirical average stopping time over the $300$ trials for each $\alpha[i]$ and  $\mathrm{Avg}_{i\in[10]}\{\cdot\}$ denotes the average over elements $i=1,\ldots,10$. We then evaluate the scaled theoretical trend on a finer grid of $1000$ values $\{\tilde{\alpha}[1],\ldots,\tilde{\alpha}[1000]\}$ such that $\{\alpha[1],\ldots,\alpha[10]\}\subseteq\{\tilde{\alpha}[1],\ldots,\tilde{\alpha}[1000]\}$. The ``Quadratic Fit'' curve is given by $\left\{\left(\tilde{\alpha}[j],\; c\cdot \frac{\log(1/\tilde{\alpha}[j])}{\Delta^2_\theta}\right)\right\}_{j=1}^{1000}$.
\end{remark}

\paragraph{Identification:}
We now empirically study the behavior of $\hat{j}_\tau$, where $\tau$ is the stopping time. We show that there is a tradeoff between $\alpha$ and the identification accuracy, emphasizing our Corollary~\ref{cor:indication}. The result for Case~1~(Gaussian) is plotted in Figure~\ref{fig:exp1_normal_hat_j}. In the figure, we introduce the notation $\texttt{ratio[k]}(\alpha[i]):=\frac{\sum_{l=1}^{300}\mathbbm{1}\{\hat{j}^{(l)}_{\tau}=k\}}{300}$ which is the empirical frequency of the output of the indicator averaged over all $300$ trials under $\alpha[i]$ and $\hat{j}^{(l)}_{\tau}$ is the output of $\hat{j}_\tau$ in $l$-th trial. For Case~2~(CIFAR-10), we observe that $\texttt{ratio[2]}(\alpha[i])=1$ for all $i\in\{1,\ldots,10\}$. Hence, we do not plot the corresponding curve for Case 2. In contrast, for Case~1 we occasionally observe $\texttt{ratio[1]}[\alpha[i]]>0$ for some $i\in[10]$ even when the alternative index $\theta=2$. In addition, such a misclassification ratio increases as $\alpha$ becomes larger. The reason is that, according to our stopping rule, a larger $\alpha$ makes earlier stopping, so our test may accumulate sufficient evidence to reject the null (while satisfying level-$\alpha$ property) but not enough evidence to reliably distinguish between the two alternative classes. This observation coincides with Proposition~\ref{theorem:identification} and Corollary~\ref{cor:indication}.

\begin{figure}[H]
    \centering
    \includegraphics[width=0.48\linewidth]{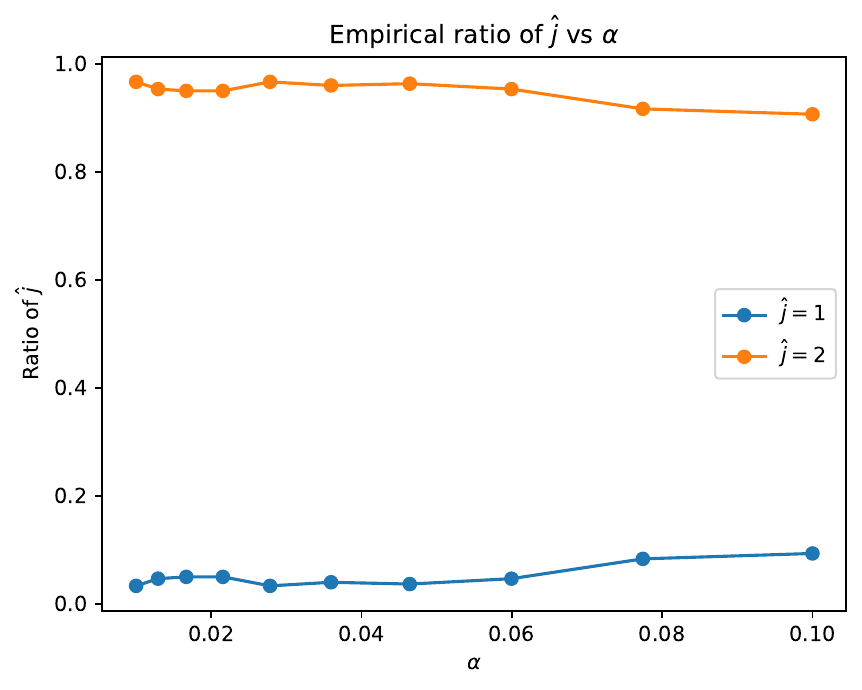}
    \caption{The empirical ratio of $\hat{j}_\tau$ vs. $\alpha$ for Case 1 under $\theta=2$. We also consider the $10$ $\alpha$-values used in simulating the average stopping time earlier and run $300$ trials at each $\alpha$-value. The blue and orange dots depict $\{(\alpha[i],\texttt{ratio[1]}(\alpha[i]))\}_{i=1}^{10}$ and $\{(\alpha[i],\texttt{ratio[2]}(\alpha[i]))\}_{i=1}^{10}$, respectively, %
    Note that we do not plot $\texttt{ratio[0]}(\alpha[i])$ since it is zero for any $i\in\{1,\ldots,10\}$.}
    \label{fig:exp1_normal_hat_j}
\end{figure}

\begin{remark}\label{remark:theoretical_result_plot}
    In our empirical results, we consider \emph{quadratic fit} and focus on the scaling trend. However, in the proof of Theorem~\ref{theorem:sequential-test}, we can obtain the exact upper bound of the expected stopping time, $\mathbb{E}_1[\tau]\leq \max\{N_0,N_1,\frac{32\log(1/\alpha)}{\Delta^2_1}\}+\frac{(L+2)\pi^2}{6}$, where
\begin{align}
    &N_0 \defined \inf \{m \geq 1: \sqrt{ (8 \log (2n^3))/n} \leq \Delta_1/2\quad\forall\,n\geq m\},\\
    &N_1 \defined \inf \lbr m \geq 1:  \frac{2 \log n}{n} \leq \frac{\Delta^2_1}{32}\quad\forall\,n\geq m \rbr.
\end{align}
One can also plot this upper bound as the theoretical result, but it is hard to see the trend under the situation that $\alpha$ is not small enough. 
\end{remark}

\subsection{Verification of Proposition~\ref{prop:robust_result}}\label{sec:exp_dist_shift}
In this section, we demonstrate the \emph{robustness} of our proposed test to small enough mismatches between training and test distributions, as implied by  Proposition~\ref{prop:robust_result}. Specifically, given the (empirical) separability gaps $\{\Delta_{\theta,m}\}_{\theta\in\calL,m\in\calL\setminus\{\theta\}}$ (recall that $\Delta_{\theta,m}=\bp_\theta[\theta]-\bp_\theta[m]$) (equivalently, given a separable classifier $g$ for the training tuple $\mathbf{P}$), we design the mismatch strength $\epsilon$ following Proposition~\ref{prop:robust_result}, then choosing one mismatched tuple $\tilde{\mathbf{P}}\in\calE_\epsilon$ to demonstrate that $0<\tilde{\Delta}_\theta<\Delta_\theta+\sqrt{2\epsilon}$ for any $\theta\in[L]$ and the induced  averaged stopping time also aligns the theoretical result.

\begin{table}[H]
\centering
\caption{Confusion matrices (in probability) of a fixed MLP classifier.}
\label{tab:exp2_toy_two_cm}
\begin{subtable}[t]{0.48\linewidth}
\centering
\caption{Evaluated on the training data (drawn from $\mathbf{P}$). Each entry of the table is a numerical version of $\mathbb{P}_{X\sim P_\theta}(g(X)=m)$, for any $\theta,m\in\calL$}
\label{tab:exp2_toy_training}
\renewcommand{\arraystretch}{1.15}
\begin{tabular}{c ccc}
\toprule
 & $g(X)=0$ & $g(X)=1$ \\
\midrule
$\theta=0$ & 0.948529 & 0.051471 \\
$\theta=1$ & 0.012195 & 0.987805  \\
\bottomrule
\end{tabular}
\end{subtable}
\hspace{0.0002\linewidth}
\begin{subtable}[t]{0.48\linewidth}
\centering
\caption{Evaluated on the testing data (drawn from $\tilde{\mathbf{P}}$). Each entry of the table is a numerical version of $\mathbb{P}_{X\sim \tilde P_\theta}(g(X)=m)$, for any $\theta,m\in\calL$.}
\label{tab:exp2_toy_testing}
\renewcommand{\arraystretch}{1.15}
\begin{tabular}{c ccc}
\toprule
 & $g(X)=0$ & $g(X)=1$ \\
\midrule
$\theta=0$ & 0.701987 & 0.298013 \\
$\theta=1$ & 0.087248 & 0.912752 \\
\bottomrule
\end{tabular}
\end{subtable}
\end{table}

\paragraph{Environments:}
Throughout this experiment, we consider the case where $L=1$. For the training tuple, we define $\mathbf{P}=\{P_0,P_1\}$, which consists of Gaussian distributions $P_0=\mathcal{N}(\mu_0, \Sigma)$ and $P_1=\mathcal{N}(\mu_1, \Sigma)$, with  $\mu_0=[-1.5, -1.0]$, $\mu_1=[1.5, 1.0]$, and $\Sigma$ is the identity matrix. Next, we train an MLP classifier with two hidden layers of sizes 10 and 5. The model is trained on $300$ samples drawn equally from $P_0$ and $P_1$ ($150$ samples each), yielding the confusion matrix shown in Table~\ref{tab:exp2_toy_training}. Based on this confusion matrix, we compute the separability gap as $\Delta_1=\bp_1[1]-\bp_0[0]\approx0.97$. The square of the minimum separability gap is therefore $\frac{1}{2}(\min\{\Delta_{1,0}, \Delta_{0,1}\})^2\approx0.4608$. Consequently, following Proposition~\ref{prop:robust_result}, we set $\epsilon=0.4608$ and use it to define the set of mismatched tuples $\calE_\epsilon$ as in~\eqref{eq:KL-mismatch-set}.

Following this, we select a specific mismatched tuple $\tilde{\mathbf{P}}=\{\tilde P_0, \tilde P_1\}\in\calE_\epsilon$ as the testing distribution to verify Proposition~\ref{prop:robust_result}. Let $\tilde{P}_0=\mathcal{N}(\tilde{\mu}_0, I)$ and $\tilde{P}_1=\mathcal{N}(\tilde{\mu}_1, I)$, where $\tilde{\mu}_0=\mu_0+\delta_0$, $\tilde{\mu}_1=\mu_1+\delta_1$, with $\delta_0=[1, 0.5]$ and $\delta_1=[-0.5, -0.5]$. The mismatch between the training and testing distributions is illustrated in Figure~\ref{fig:exp2_toy_vis} using empirical samples.  Since the testing samples exhibit greater overlap than the corresponding training samples, it is more challenging for  the MLP classifier to distinguish between $\tilde{P}_0$ and $\tilde{P}_1$.

Next, we verify that $\tilde{\mathbf{P}}$ belongs to $\calE_\epsilon$, which is defined by the specified $\epsilon$ and the trained MLP classifier. The confusion matrix for the classifier evaluated on $\tilde{\mathbf{P}}$ is provided in Table~\ref{tab:exp2_toy_testing}, where one can note that the power of separability of the classifier under $\tilde{\mathbf{P}}$ is weaker than that under the training distribution $\mathbf{P}$. We then compute $\max\{D(\bp_0\|\tilde\bp_0), D(\bp_1\|\tilde\bp_1)\}\approx0.1951 < \epsilon$ (where the KL-divergence uses the natural logarithm), confirming that $\tilde{\mathbf{P}}\in\calE_\epsilon$. Moreover, Table~\ref{tab:exp2_toy_testing} also shows that $\tilde\bp_\theta[\theta]>\tilde\bp_\theta[0]$ for any $\theta\in\calL$. This indicates that the MLP classifier trained on the original $\mathbf{P}$ remains separable for the shifted $\tilde{\mathbf{P}}$. Finally, we observe $\tilde\Delta_1=\tilde\bp_1[1]-\tilde\bp_1[0]\approx0.83$, which successfully verifies the theoretical bound $0<\tilde\Delta_1\leq \Delta_1 + \sqrt{2\epsilon}$.

\begin{figure}[t]
    \centering
    \includegraphics[width=0.48\linewidth]{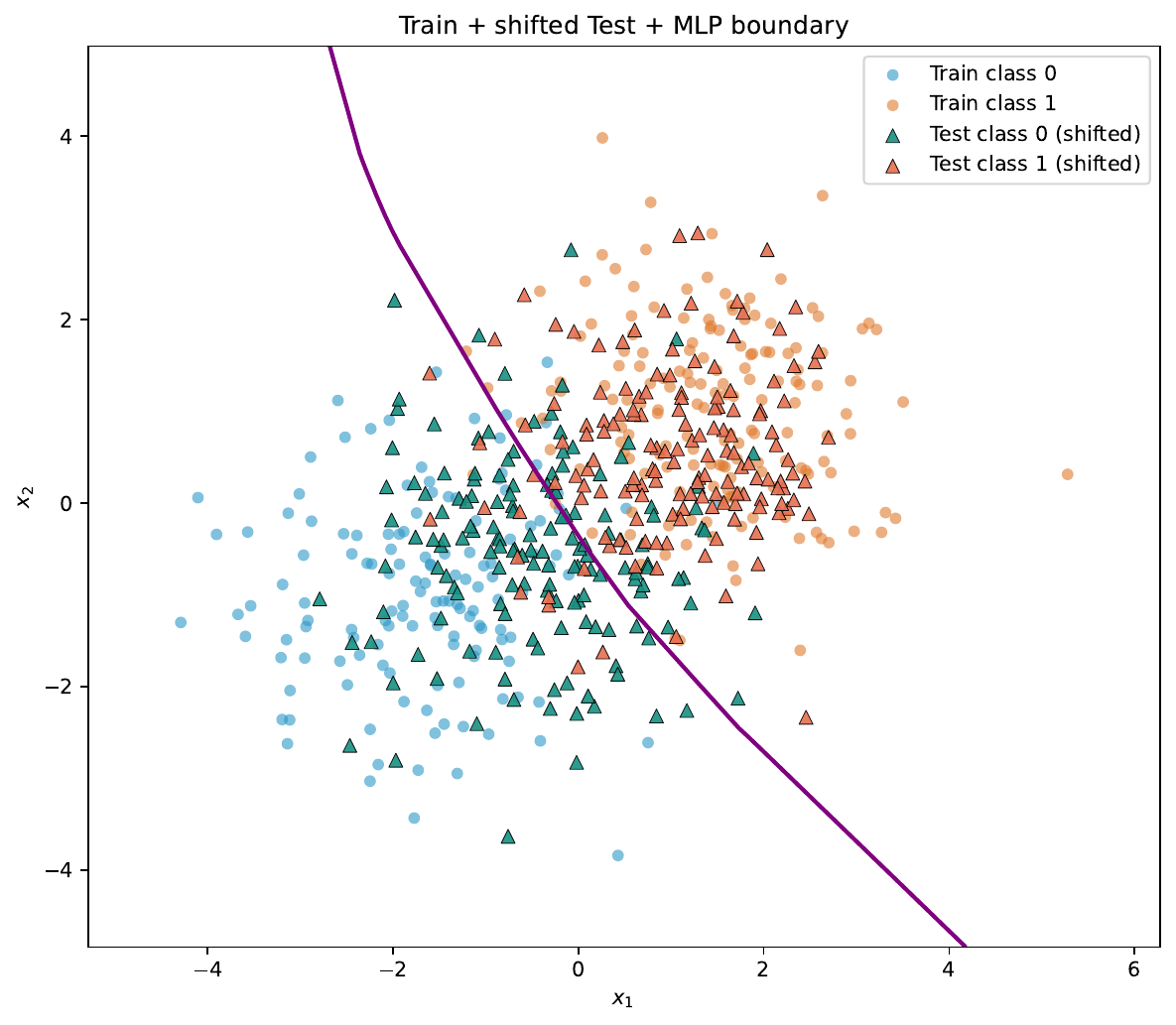}
    \caption{We visualize the distribution shift in this figure. The orange and blue points represent training samples drawn from $P_0$ and $P_1$, respectively, while the red and green rectangles represent testing samples drawn from the mean-shifted distributions $\tilde{P}_0$ and $\tilde{P}_1$. The purple curve shows the decision boundary of the MLP classifier trained on  samples drawn from $(P_0, P_1)$.}
    \label{fig:exp2_toy_vis}
\end{figure}

\paragraph{Average stopping time:} Following the environment setup, we verify the expected stopping time shown in Proposition~\ref{prop:robust_result} under the chosen $\tilde{\mathbf{P}}$ and the MLP classifier trained on the training tuple $\mathbf{P}$. The results are plotted in Figure~\ref{fig:exp2_toy_stopping_time_comp}, where we compare the numerical expected stopping times of the non-shifted and shifted testing data distributions. Specifically, the trend of the theoretical expected stopping time matches the empirical average (following the argument of the quadratic fit in Section~\ref{sec:verify_thm_sht}). However, since the shifted distributions share more overlap around the classifier's decision boundary (see Figure~\ref{fig:exp2_toy_vis}), the obtained average stopping time is strictly larger than its non-shifted counterpart. The detailed implementation for this subsection is provided in \texttt{expt\_NormalModel\_visualization.py} in our GitHub repository.

\begin{figure}[H]
    \centering
    \begin{subfigure}[t]{0.48\linewidth}
        \centering
        \includegraphics[width=\linewidth]{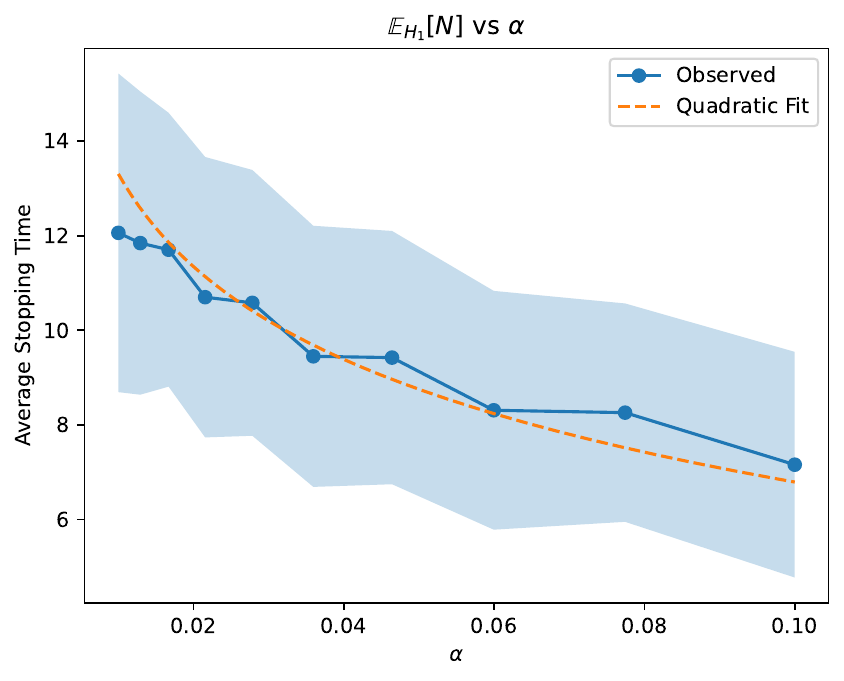}
        \caption{Testing data are drawn from the non-shifted tuple $\mathbf{P}$, which is exactly the training data distributions.}
        \label{fig:exp2_toy_stop_time}
    \end{subfigure}
    \hfill
    \begin{subfigure}[t]{0.48\linewidth}
        \centering
        \includegraphics[width=\linewidth]{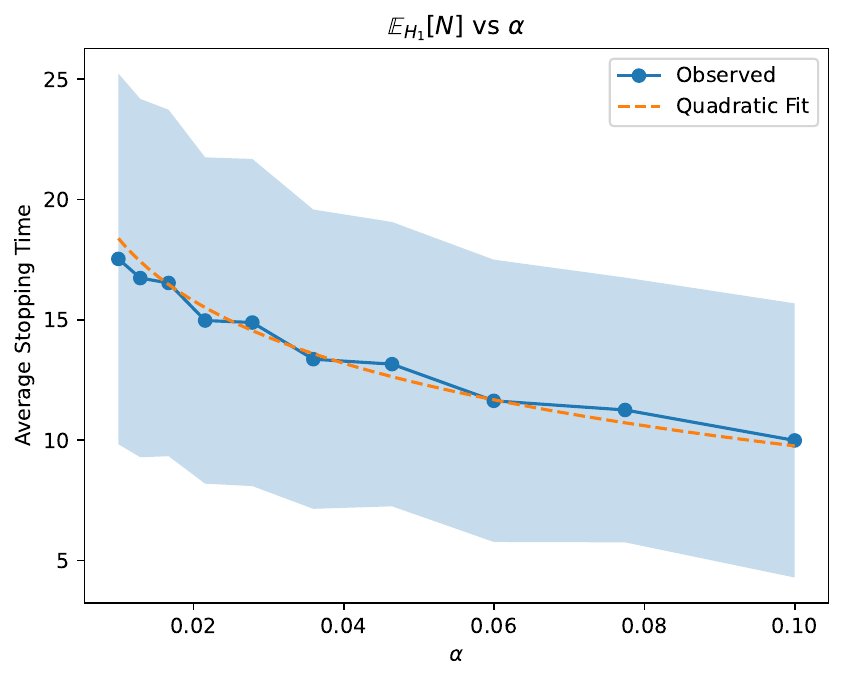}
        \caption{Testing data is drawn from the mean-shifted tuple $\tilde{\mathbf{P}}$, which is shifted from the training data distributions.}
        \label{fig:exp2_toy_shifted_stop_time}
    \end{subfigure}
    \caption{Average stopping time comparison, where both utilize the MLP classifier trained on the non-shifted tuple $\mathbf{P}$. We plot these two figures using the same $\alpha$ values and quadratic fitting method as in Figure~\ref{fig:exp1}. However, the quadratic fit in Figure~\ref{fig:exp2_toy_stop_time} represents the theoretical scaling trend derived from Table~\ref{tab:exp2_toy_training}, whereas its counterpart in Figure~\ref{fig:exp2_toy_shifted_stop_time} is based on Table~\ref{tab:exp2_toy_testing}.}
    \label{fig:exp2_toy_stopping_time_comp}
\end{figure}

\subsection{Verification of Theorem~\ref{theorem:change-detection}}\label{sec:verify_thm_scd}
In this section, we demonstrate the empirical results of our proposed SCD test from Section~\ref{sec:extension_scd}. For the data-generating model, the pre-change and post-change distributions are, respectively, $P_0$ and $P_1$, as defined in Section~\ref{sec:exp_dist_shift}. In this change-point detection scheme, the observed samples are i.i.d. drawn from $P_0$ for time steps prior to a change-point (denoted as $T$). From time step $T$ onward, the data distribution becomes $P_1$. In the following experiment, we set $T=10$.
Moreover, for the classifier, we also use an MLP classifier with two hidden layers of sizes 10 and 5 trained on $10000$ samples uniformly drawn from $P_0$ and $P_1$ ($N=5000$). The confusion matrix is shown in Table~\ref{tab:exp4_Gaussian}. Then, the empirical expected detection delay is shown in Figure~\ref{fig:exp4_delat_Gaussian}. We can notice that the empirical result matches the scaling trend of the expected delay in Theorem~\ref{theorem:change-detection}.

\begin{table}[H]
\centering
\caption{Confusion matrix (in probability) of the classifier. Each entry of the table is a numerical version of $\mathbb{P}_{X\sim P_\theta}(g(X)=m)$, for any $\theta,m\in\calL$.}
\label{tab:exp4_Gaussian}
\begin{tabular}{c cc}
\toprule
& {$g(X)=0$} & {$g(X)=1$} \\
\midrule
$\theta=0$ & 0.933413 & 0.066587\\
$\theta=1$ & 0.063393 & 0.936607 \\
\bottomrule
\end{tabular}
\end{table}

\begin{figure}[htb]
\centering
\includegraphics[width=0.48\linewidth]{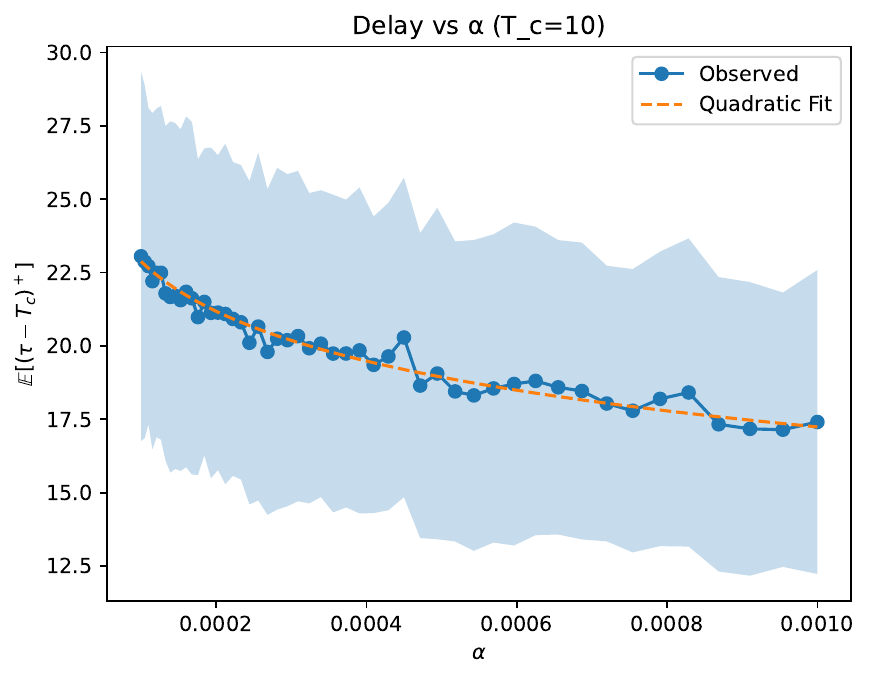}
\caption{Average expected delay vs. $\alpha$. We consider $\alpha$ values ranging from $10^{-4}$ to $10^{-3}$, using \texttt{numpy geomspace} to generate $50$ logarithmically spaced levels. For each $\alpha$, we run $500$ independent trials and plot the average delay (blue dots). The orange dashed curve shows the corresponding quadratic fit, computed in the same manner as in Section~\ref{sec:verify_thm_sht}, with the expected stopping time replaced by the expected detection delay. Note that $T_c$ in the figure corresponds to the change point $T$ used in this paper.}
\label{fig:exp4_delat_Gaussian}
\end{figure}

\begin{remark}
    We do not show empirical results for the expected stopping time under the no-change regime ($T=\infty$). This case follows as a straightforward consequence of the analysis in~\cite{shin2022detectors}, and obtaining a meaningful simulation ($\alpha$ is small) would require long runs due to the rarity of false alarms at small $\alpha$.
\end{remark}

\subsection{Mixture of Classifiers}\label{sec:exp_mixture}
Recall that in Section~\ref{sec:main-results}, we mentioned that the expected stopping time is highly related to the separability gap $\Delta_\theta$ for $\theta\in\calL$, which is the separability of the classifier under the ground truth distribution tuple $\mathbf{P}\in\mathcal{P}_{\text{sep}}$. This motivates us to consider using \emph{multiple} classifiers instead of just one. This modification may introduce the gain of heterogeneity. In this section, we provide empirical evidence to demonstrate that leveraging multiple classifiers benefits the expected stopping time.

The data distribution tuple and all other experimental parameters (including the number of training samples and the number of trials, etc.) are the same as the Gaussian case in Section~\ref{sec:verify_thm_sht}. The difference is that we consider two MLP classifiers, both retrained for this experiment. The first classifier, denoted by $g_1$, is trained with two hidden layers of sizes 64 and 32. The second classifier, denoted by $g_2$, is a \emph{stronger} model trained with two hidden layers of sizes 128 and 64. The confusion matrices of $g_1$ and $g_2$ are shown in Table~\ref{tab:exp3_two_cm}. One can notice that $g_2$ achieves a higher probability $P_{\theta}\!\bigl(g_2(X)=\theta\bigr)$ for every $\theta\in\calL$ than $g_1$ does. Moreover, if we define $\Delta_\theta^{g_i}:=\mathbb{P}_{X\sim\theta}(g_i(X)=\theta)-\max_{m\in\calL\setminus\{\theta\}}\mathbb{P}_{X\sim\theta}(g_i(X)=m)$ for any $\theta\in\calL$, then according to the confusion matrices, we get $\Delta_\theta^{g_2}>\Delta_\theta^{g_1}$ for all $\theta\in\calL$. This indicates that $g_2$ has a stronger separability under the same ground truth distribution tuple. %

\begin{table}[H]
\centering
\caption{Confusion matrices (in probability) of classifiers.}
\label{tab:exp3_two_cm}
\begin{subtable}[t]{0.48\linewidth}
\centering
\caption{Classifier $g_2$. Each entry of the table is a numerical version of $\mathbb{P}_{X\sim P_\theta}(g_2(X)=m)$, for any $\theta,m\in\calL$}
\label{tab:exp3_g2}
\renewcommand{\arraystretch}{1.15}
\begin{tabular}{c ccc}
\toprule
 & $g_2(X)=0$ & $g_2(X)=1$ & $g_2(X)=2$ \\
\midrule
$\theta=0$ & 0.635678 & 0.132508 & 0.231814 \\
$\theta=1$ & 0.167847 & 0.594985 & 0.237168 \\
$\theta=2$ & 0.149833 & 0.123749 & 0.726418 \\
\bottomrule
\end{tabular}
\end{subtable}
\hfill
\begin{subtable}[t]{0.48\linewidth}
\centering
\caption{Classifier $g_1$. Each entry of the table is a numerical version of $\mathbb{P}_{X\sim P_\theta}(g_1(X)=m)$, for any $\theta,m\in\calL$}
\label{tab:exp3_g1}
\renewcommand{\arraystretch}{1.15}
\begin{tabular}{c ccc}
\toprule
 & $g_1(X)=0$ & $g_1(X)=1$ & $g_1(X)=2$ \\
\midrule
$\theta=0$ & 0.482644 & 0.230909 & 0.286447\\
$\theta=1$ & 0.198230 & 0.515929 & 0.285841 \\
$\theta=2$ & 0.200485 & 0.233546 & 0.565969 \\
\bottomrule
\end{tabular}
\end{subtable}
\end{table}

However, we usually do not know which classifier has the strongest separability since the ground truth tuple is unknown. 
Then, one way to utilize the two classifiers is incorporating an dynamic \emph{weighting} on the e-process constructed based on these two classifiers. Specifically, let $\{W_{n,1}\}_{n\ge0}$ and $\{W_{n,2}\}_{n\ge0}$ be the $e$-process defined as follows: Let $\hat{j}_{t,k} = \argmax_{j \in \calL} \; \hat{\bp}_{t-1,k}[j]$ and $\hat{\bp}_{n,k} = (1/n) \sum_{i=1}^n\delta_{g_k(X_i)}$ for $k\in\{1,2\}$ and
\begin{align}
        &W_{0,1} = W_{0,2} =  1, \qquad W_{n,k}:=\int_{-1}^0 \lp \prod_{t=1}^n \lp 1 + \lambda \times \lp \boldsymbol{1}_{g_k(X_t)=0} - \boldsymbol{1}_{g_k(X_t)=\hat{j}_{t,k}} \rp \rp \rp d\lambda.
\end{align}
Then, we can define the dynamic weighting sequence $\{w_n\}_{n\ge0}$, $w_n=[w_{n,1}, w_{n,2}]$, which is (adaptively) determined by the past information, resulting the mixture e-process $\{W^{\text{mix}}_n\}_{n\ge1}$ defined as $W^{\text{mix}}_n=w_{n,1}W_{n,1}+w_{n,2}W_{n,2}$ for each $n\ge0$. The goal is to make $\{w_n\}_{n\ge0}$ concentrate to $w_\infty=[0,1]$ in this case since the classifier $g_2$ has a stronger separability.

We first argue the validity of this mixture method. Since the sum of non-negative supermartingales is still a non-negative supermartingale, we can directly apply Ville's inequality to establish the level-$\alpha$ property of $\{W^{\text{mix}}_n\}_{n\ge0}$. Moreover, regarding the power-one property, since both classifiers are separable under the ground truth tuple, it is intuitive that the process $\{W^{\text{mix}}_n\}_{n\ge0}$ still satisfies the power-one property under any weight.

However, designing the dynamic weighting sequence $\{w_n\}_{n\ge0}$ is a nontrivial task. It connects to the classic problem known as the \emph{exploration-exploitation tradeoff} in the multi-armed bandit literature \cite{Slivkins2024}. A high-level intuition is that we can regard the term $\lp \boldsymbol{1}_{g_k(X_t)=0} - \boldsymbol{1}_{g_k(X_t)=\hat{j}_{t,k}} \rp$ as an instantaneous reward and apply a bandit algorithm since for any fixed $\hat{j}_{t,k}\neq0$, the expected value of this term under $g_2$ is greater than that under $g_1$, due to the stronger separability that $g_2$ owns. Nevertheless, we leave the detailed analysis as a future direction. 

Here, we demonstrate that the benefit of using multiple classifiers exists. Assume an \emph{oracle} weighting mechanism sets $w_1=0.1$ and $w_2=0.9$. The empirical results are plotted in Figure~\ref{fig:exp3}. We observe that the average stopping time under the mixture method (blue dot curve) is smaller than that achieved by a single classifier (orange rectangular curve). Therefore, a weighting mechanism exists that successfully leverages classifier heterogeneity. As noted, designing a dynamic weighting scheme is not addressed in this paper. However, we have provided a clear direction for future work.

\begin{figure}[htb]
\centering
\includegraphics[width=0.48\linewidth]{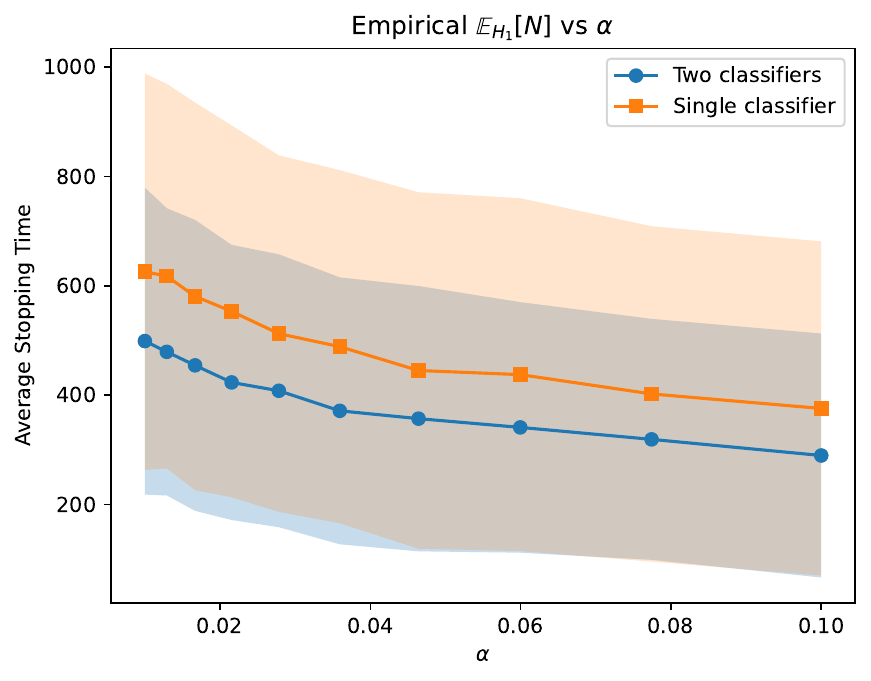}
\caption{Average expected stopping time vs. desired level-$\alpha$ values. The $\alpha$ values are taken to be the same as those defined in Section~\ref{sec:verify_thm_sht}. For each $\alpha$, we run $300$ independent trials and plot the average stopping time for both scenarios (for two classifiers represented by blue dots and for the single classifier represented by orange rectangles). We can note that the mixture method constructed by two classifiers outperforms the single classifier.}
\label{fig:exp3}
\end{figure}

\section{Conclusion and Future Work}
\label{sec:conclusion}
We propose a general framework for constructing nonparametric sequential power-one tests for problems where the null and alternatives are defined indirectly through offline data. Our approach proceeds by constructing an e-process based on an offline-trained separable classifier that yields an anytime-valid test satisfying the level-$\alpha$ and power-one properties. Furthermore, we provide a non-asymptotic analysis of the expected stopping time under the alternative, showing that the separability of the classifier plays an important role. We also investigate how the training sample size and the classifier family affect the attainability of separable classifiers and the achievability. Specifically, we establish sufficient and necessary conditions for the required training sample size and analyze its impact, along with the classifier family, on the tail probability of the stopping time. Finally, we extend our stopping rule and theoretical guarantees to the training-testing mismatch problem, showing robustness when the mismatch is bounded by the classifier's separability, and demonstrate its applicability to sequential change-point detection. 

Empirically, we use synthetic and real data to verify that for the proposed stopping rule and its extensions, the observed average stopping times (or average delay for sequential change-point detection) match their theoretical scaling trends. We also verify the robustness of this classifier setting. Regarding future work, it remains unclear whether our proposed test is optimal in general (assuming a fixed classifier and a fixed distribution tuple with no mismatch), as we  only establish asymptotic optimality~(as $\alpha \downarrow 0$) for specific problem instances. Moreover, the impact of the classifier on the overall performance is not yet fully understood since we do not directly investigate the complexity of $\calG$ in the minimax lower bound. Finally, a thorough theoretical and empirical study of the dynamic weighting discussed in Section~\ref{sec:exp_mixture} is another interesting direction for future work.

\bibliographystyle{IEEEtran}
\bibliography{ref}
\newpage 
\appendix

\section{Additional Background}
\label{appendix:additional-background}

\begin{fact}[Ville's Inequality~\cite{Ville1939}]
    \label{fact:villes-inequality}
    Suppose $\{W_n: n \geq 0\}$ is a nonnegative supermartingale. For any $\alpha \in (0, 1)$, we have the following 
    \begin{align}
        \mathbb{P}\lp \exists n \geq 0: W_n \geq 1/\alpha \rp \leq \frac{\mathbb{E}[W_0]}{\alpha}. 
    \end{align}
\end{fact}

\section{Proof of Theorem~\ref{theorem:sequential-test}}
\label{proof:sequential-test}
In this section, we prove that our proposed stopping rule defined in Section~\ref{sec:proposed_test} satisfies the level-$\alpha$ power one criterion and analyze its resulting expected stopping time. We assume that the ground-truth distribution tuple $\mathbf{P}\in\mathcal{P}_{\text{sep}}$ and its corresponding separable classifier are given (while the decision maker does not know the distribution tuple $\mathbf{P}$). Moreover, the training and testing distributions have no mismatch. We derive the proofs through the notations defined in the beginning of Section~\ref{sec:proposed_test}.

\paragraph{Under $H_0$.} Under the null, the result follows directly from the fact that $\{W_n: n \geq 0\}$ is a nonnegative supermartingale with an initial value of $1$.

\paragraph{Under $H_1$.} We directly show that the expected stopping time is finite, which implies that the power-one property is satisfied. Throughout this proof, for notational simplicity, we use $\{Z_i: i \geq 1\}$ to denote the process $\{g(X_i): i \geq 1\}$. Since $\tau$ is a positive integer valued random variable, our starting point is, for any $\theta\in\{1,\ldots,L\}$,
\begin{align}
    \mathbb{E}_\theta[\tau] & = \sum_{n = 1}^{\infty} \mathbb{P}_\theta\lp \tau \geq n \rp  \leq \sum_{n = 1}^{\infty} \mathbb{P}_\theta\lp \log W_n \geq  \log(1/\alpha)\rp,  \label{eq:sequential-test-proof-1}
\end{align}
where the last inequality holds true since $\{\tau\geq n\}\subseteq\{\log W_n \geq  \log(1/\alpha)\}$.
Before proceeding further, we need to introduce some new terms which are used to split the probability of $\{\log W_n \geq  \log(1/\alpha)\}$. For any $n \geq 1$, introduce the following events: 
\begin{align}
    E_n &= \cap_{j=0}^L E_{n,j}, \quad \text{with} \quad E_{n,j} \defined \left\{ \forall 1 \leq t \leq n: \left\lvert \hat{\bp}_t[j] - \bp_\theta[j] \right\rvert \leq b^{(n)}_t \right\}, \quad \text{for}\quad b^{(n)}_t \defined \sqrt{\frac{\log 2n^3}{2t}}, \\
    F_n &= \left\{ \left\lvert \frac{1}{n} \sum_{t=1}^n \lp \boldsymbol{1}_{Z_t=0} - \boldsymbol{1}_{Z_t=\hat{j}_t}\rp - \lp \bp_\theta[0] - \bp_\theta[\hat{j}_t] \rp \right\rvert \leq c_n  \right\}, \quad \text{for} \quad c_n = \sqrt{\frac{2 \log 2n^2}{n}}. 
\end{align}
Our first lemma establishes that these two events are of sufficiently large probability. 
\begin{lemma}
    \label{lemma:sequential-test-1} 
    For all $n \geq 2$, $\theta\in\{1,\ldots,L\}$, we have 
    \begin{align}
        \mathbb{P}_\theta\lp E_n \rp \geq 1 - \frac{L+1}{n^2}, \quad \text{and} \quad 
        \mathbb{P}_\theta\lp F_n \rp \geq 1 - \frac{1}{n^2}. 
    \end{align}
    Crucially, this implies 
    \begin{align}
        &\sum_{n \geq 1} \mathbb{P}_\theta\lp E_n^c \rp \leq \sum_{n \geq 1} \frac{L+1}{n^2} = \frac{(L+1)\pi^2}{6} < \infty, 
        \quad \text{and} \quad 
        \sum_{n \geq 1} \mathbb{P}_\theta\lp F_n^c \rp \leq \sum_{n \geq 1} \frac{1}{n^2} = \frac{\pi^2}{6} < \infty. 
    \end{align}
\end{lemma}
The proof of this Lemma is deferred to~\Cref{proof:sequential-test-1}.

Continuing with~\eqref{eq:sequential-test-proof-1}, we can now write the following: 
\begin{align}
    \mathbb{E}_\theta[\tau] & \leq \sum_{n \geq 1} \mathbb{P}_\theta\lp \lbr \log W_n < \log(1/\alpha) \rbr \cap (E_n \cap F_n)\rp + \mathbb{P}_\theta\lp E_n^c \rp + \mathbb{P}_\theta(F_n^c) \\
    & \leq \sum_{n \geq 1} \mathbb{P}_\theta\lp \lbr \log W_n < \log(1/\alpha) \rbr \cap (E_n \cap F_n)\rp + \frac{(L+2)\pi^2}{6}. \label{eq:sequential-test-proof-2}
\end{align}
In the second inequality, we have used Lemma~\ref{lemma:sequential-test-1} to bound the last two terms. This inequality implies that we need to analyze the behavior of $\log W_n$ under the events $E_n \cap F_n$. To do so, we begin with the following lemma. 
\begin{lemma}
    \label{lemma:sequential-test-2} Let $Y_i$ denote $\boldsymbol{1}_{Z_i = 0} - \boldsymbol{1}_{Z_i=\hat{j}_i}$, where $\hat{j}_i$ is selected according to the rule described in~\eqref{eq:j-selection}. Let $\Delta_\theta[j]$ denote the value $\bp_\theta[j]-\bp_\theta[0]$ for any $j \in \calL$, and observe that for any $i \geq 1$, we have $\mathbb{E}[Y_i|\calF_{i-1}] = - \Delta_\theta[\hat{j}_i]$.  Then, under the event $E_n$, we have the following: $\forall\,\theta\in\{1,\ldots,L\}$
    \begin{align}
        \frac{1}{n} \sum_{i=1}^n \Delta_\theta[\hat{j}_i] \geq \Delta_\theta - \sqrt{\frac{2 \log2 n^3}{n}}, 
    \end{align}
    where $\Delta_\theta = \max_{j \in \calL} \bp_\theta[j] - \bp_\theta[0]=\bp_\theta[\theta]-\bp_\theta[0]$. 
\end{lemma}
The proof of this statement is in Appendix~\ref{proof:sequential-test-2}. 

Our next result compares the average conditional expectations of $(Y_1, \ldots, Y_n)$ to $\Delta_\theta$. 
\begin{lemma}
    \label{lemma:sequential-test-3} Following the same notations as in Lemma~\ref{lemma:sequential-test-2}, under the event $E_n \cap F_n$, we have 
    \begin{align}
        \frac{1}{n} \sum_{i=1}^n Y_i \leq - \Delta_\theta + \sqrt{\frac{2 \log 2 n^3}{n}} +  \sqrt{\frac{2 \log 2n^2}{n}} \leq - \Delta_\theta + \sqrt{\frac{8 \log 2n^3}{n}}\quad\forall\,\theta\in\{1,\ldots,L\}.
    \end{align}
    As a result,  we have 
    \begin{align}
        \frac{1}{n} \sum_{i=1}^n Y_i \leq - \frac{\Delta_\theta}{2}, \quad \text{for} \quad  n \geq N_0 \defined \inf \{m \geq 1: \sqrt{ (8 \log (2n^3))/n} \leq \Delta_\theta/2, \forall\,n\ge m\}  = \mc{O}\lp \frac{\log(1/\Delta_\theta)}{\Delta^2_\theta}\rp. \label{eq:sequential-test-proof-5}
    \end{align}
\end{lemma}
The proof of this result in Appendix~\ref{proof:sequential-test-3}.
We need one final technical result that gives us a lower bound on $\log W_n$, before we can combine everything to conclude the proof. 
\begin{lemma}
    \label{lemma:sequential-test-4} For $n \geq N_0$, with $N_0$ defined in~\eqref{eq:sequential-test-proof-5}, we have the following under the event $E_n \cap F_n$: $\forall\,\theta\in\{1,\ldots,L\}$
    \begin{align}
        \frac{\log W_n}{n} \geq  \frac{\Delta^2_\theta}{16} - \frac{2 \log n}{n}. 
    \end{align}
    This immediately leads to the following simplification: $\forall\,\theta\in\{1,\ldots,L\}$
    \begin{align}
        \frac{\log W_n}{n} \geq \frac{\Delta^2_\theta}{32}, \quad \text{for} \quad n \geq N_1 \defined \inf \lbr m \geq 1:  \frac{2 \log n}{n} \leq \frac{\Delta^2_\theta}{32},\forall\,n\geq m \rbr = \mc{O}\lp \frac{\log(1/\Delta_\theta)}{\Delta^2_\theta}\rp.  \label{eq:sequential-test-proof-7}
    \end{align}
\end{lemma}
The proof of this result in Appendix~\ref{proof:sequential-test-4}.

\sloppy We now have all the components to complete the proof. Starting with~\eqref{eq:sequential-test-proof-2}, we have with $T \defined \max\{N_0, N_1, 32\log(1/\alpha)/\Delta^2_\theta\}$, where $N_0$ and $N_1$ were introduced in~\eqref{eq:sequential-test-proof-5} and~\eqref{eq:sequential-test-proof-7} respectively.  
\begin{align}
    \mathbb{E}_\theta[\tau] 
    & \leq \sum_{n \geq 1} \mathbb{P}\lp \lbr \log W_n < \log(1/\alpha) \rbr \cap (E_n \cap F_n)\rp + \frac{(L+2)\pi^2}{6} \\
    & \leq T +  \sum_{n \geq T} \mathbb{P}\lp \lbr \log W_n < \log(1/\alpha) \rbr \cap (E_n \cap F_n)\rp + \frac{(L+2)\pi^2}{6} \\
    & \leq T +  \sum_{n \geq T} \mathbb{P}\lp \lbr \frac{\Delta_\theta^2}{32} < \frac{\log(1/\alpha)}{n} \rbr \cap (E_n \cap F_n)\rp + \frac{(L+2)\pi^2}{6} \\
    & = T + \frac{(L+2)\pi^2}{6},
\end{align}
where the third inequality applies Lemma~\ref{lemma:sequential-test-4} and the last equality uses the fact that for all $n \geq T$, we have $\log(1/\alpha)/n \leq \Delta^2/32$ by the definition of $T$. 

Hence, we have shown that 
\begin{align}
    \mathbb{E}_\theta[\tau] = T + \mc{O}(1) = \mc{O}\lp \frac{\log\!\big(1/(\alpha \Delta_\theta)\big)}{\Delta_\theta^2} \rp, 
    \quad \text{where} \quad 
    \Delta_\theta = \max_{j \in \calL} \bp_\theta[j] - \bp_\theta[0] = \bp_\theta[\theta]-\bp_\theta[0]. 
\end{align}

\subsection{Proof of Auxiliary Lemmas}
\label{proof:sequential-test-lemmas}

\subsubsection{Proof of Lemma~\ref{lemma:sequential-test-1}}
\label{proof:sequential-test-1}

\paragraph{Event $E_n$.} It suffices to show that for any $j \in \calL$, we have $\mathbb{P}(E_{n,j}^c) \leq 1/n^2$, since the required result follows by union bound. To show $E_{n,j}$, we observe that since $\hat{\bp}_t[j] = \frac{1}{t} \sum_{i=1}^t \boldsymbol{1}_{Z_i=j}$ is the average of bounded \iid random variables, we have by an application of Hoeffding's inequality: 
\begin{align}
    \mathbb{P}_\theta\lp \left\lvert \frac{1}{t} \sum_{i=1}^t \boldsymbol{1}_{Z_i=j} - \bp_\theta[j] \right\rvert \leq  \sqrt{\frac{ \log (2n^3)}{2t}}  \rp \geq 1 - \frac{1}{n^3}. 
\end{align}
Denoting $\sqrt{\log(2n^3)/2t}$ with $b^{(n)}_t$, and by taking a union bound for $1 \leq t \leq n$, we have 
\begin{align}
    &\mathbb{P}\lp \forall 1 \leq t \leq n:  \left\lvert \frac{1}{t} \sum_{i=1}^t \boldsymbol{1}_{Z_i=j} - \bp_\theta[j] \right\rvert \leq b^{(n)}_t\rp \geq 1 - \sum_{t=1}^n \frac{1}{n^3}= 1 - \frac{1}{n^2}.
\end{align}

\paragraph{Event $F_n$.}
We can apply Azuma's inequality to bound this term. Introduce the following notations: $Y_i = \boldsymbol{1}_{Z_i=0} - \boldsymbol{1}_{Z_i=\hat{j}_i}$, $S_n = \sum_{i=1}^n Y_i$, and $V_n = S_n - \sum_{i=1}^n \mathbb{E}_\theta\!\big[Y_i \mid \calF_{i-1}\big]$.
We have 
\begin{align}
    \mathbb{P}_\theta\lp F_n^c\rp = \mathbb{P}_\theta\lp \left\lvert S_n - \mathbb{E}_\theta\!\big[Y_i \mid \calF_{i-1}\big]\right\rvert>nc_n\rp = \mathbb{P}_\theta\lp |V_n| >nc_n\rp.
\end{align}
Then, since $\{V_n, \calF_{n}\}_{n\ge0}$ is a martingale process by
\begin{align}
    \mathbb{E}_\theta[V_n|\calF_{n-1}] - V_{n-1}&=\mathbb{E}_\theta[V_n-V_{n-1}|\calF_{n-1}]=\mathbb{E}_\theta[S_n-S_{n-1} - \mathbb{E}_\theta[Y_n\mid\calF_{n-1}]|\calF_{n-1}]\\
    &=\mathbb{E}_\theta[Y_n\mid\calF_{n-1}] -\mathbb{E}_\theta[Y_n\mid\calF_{n-1}\mid F_{i-1}]|\calF_{n-1}]=0,
\end{align}
and $|V_n-V_{n-1}|\leq 1$, we can apply Azuma's inequality: 
\begin{align}
    \mathbb{P}_\theta\lp F_{n}^c\rp=\mathbb{P}_\theta\lp|V_n|>nc_n\rp\leq2\exp\lp-\frac{nc_n^2}{2}\rp=\frac{1}{n^2}.
\end{align}

\subsubsection{Proof of Lemma~\ref{lemma:sequential-test-2}}
\label{proof:sequential-test-2}
The proof of this statement is modeled according to the classical regret analysis techniques for the upper confidence bound~(UCB) strategy for multiarmed bandits. In particular, if $\hat{j}_i=j\neq \theta$ at time $i$, then we must have, under $E_n$, $\forall\,i\leq n+1$,
\begin{align}
    \hat{\bp}_{i-1}[j]  > \hat{\bp}_{i-1}[\theta]  \implies \hat{\bp}_{i-1}[j] + b^{(n)}_{i-1} > \hat{\bp}_{i-1}[\theta] + b^{(n)}_{i-1} \geq \bp_\theta[\theta]  \label{eq:sequential-test-proof-3}
\end{align}
The first inequality is due to the selection rule~\eqref{eq:j-selection}, while the third inequality says that under the event $E_n$, $\bp_\theta[\theta]$ is not larger than the upper bound. Using these simple observations, we have for $i \geq 2$:
\begin{align}
    \Delta_\theta - \Delta_\theta[\hat{j}_i] &= \bp_\theta[\theta] - \bp_\theta[\hat{j}_i] \\
    & \leq \hat{\bp}_{i-1}[\hat{j}_i] + b^{(n)}_{i-1, \hat{j}_i}- \bp[\hat{j}_i] &&  (\text{by equation~\eqref{eq:sequential-test-proof-3}}) \\
    & \leq \hat{\bp}_{i-1}[\hat{j}_i] + b^{(n)}_{i-1} - \hat{\bp}_{i-1}[\hat{j}_i] + b^{(n)}_{i-1} && (\text{event $E_n$}) \\
    & = 2 \sqrt{\frac{\log 2n^3 }{2(i-1)}} && (\text{definition of $b^{(n)}_{i-1}$})  \\
    & =  \sqrt{\frac{2 \log 2n^3}{i-1}}. \label{eq:sequential-test-proof-4} 
\end{align}
Now, summing this from $i=1$ to $n$, we get 
\begin{align}
    \sum_{i=1}^n \Delta_\theta - \Delta_\theta[\hat{j}_i] & \leq 1 + \sum_{i=2}^n \sqrt{\frac{2 \log 2n^3}{i-1}} \leq \sqrt{2 \log 2n^3} \sum_{i=1}^n \frac{1}{\sqrt{i}} \leq \sqrt{2 n \log 2n^3}, 
\end{align}
where in the second inequality, we have used the fact that $\sqrt{2 \log 2n^3} > 1$ for $n \geq 2$, and the third inequality follows from the bound $\sum_{i=1}^n 1/\sqrt{i} \leq \sqrt{n}$.  Dividing throughout by $n$, gives us the required inequality 
\begin{align}
    \frac{1}{n} \sum_{i=1}^n \Delta_\theta[\hat{j}_i] \geq \Delta_\theta - \sqrt{\frac{2 \log 2n^3}{n}}. 
\end{align}
\subsubsection{Proof of Lemma~\ref{lemma:sequential-test-3}}
\label{proof:sequential-test-3}
This is a direct consequence of the previous lemma, and the definition of the event $F_n$. In particular, recall that under the event $F_n$, we have 
\begin{align}
     \lv \frac{1}{n} \sum_{i=1}^n Y_i + \Delta_\theta[\hat{j}_i] \rv \leq c_n \defined \sqrt{\frac{2 \log 2n^2}{n}}.  
\end{align}
This fact, combined with Lemma~\ref{lemma:sequential-test-2} implies that 
\begin{align}
    \frac{1}{n} \sum_{i=1}^n Y_i \leq -\frac{1}{n} \sum_{i=1}^n \Delta_\theta[\hat{j}_i] + \sqrt{\frac{2 \log 2n^2}{n}} \leq - \Delta_\theta + \sqrt{\frac{2 \log 2n^3}{n}} + \sqrt{\frac{2 \log 2n^2}{n}} \leq -\Delta_\theta + \sqrt{\frac{8 \log 2n^3}{n}}, 
\end{align}
where we get the last inequality by  loosely upper bounding $\sqrt{ (2 \log 2n^2)/n}$ with $\sqrt{(2 \log 2n^3)/n}$. Finally, $N_0$ is defined as the smallest value of $n$ at which the term $\sqrt{(8 \log 2n^3)/n}$ goes below $\Delta_\theta/2$, and thus, for all $n \geq N_0$, we have $\frac{1}{n} \sum_{i=1}^n Y_i \leq -\Delta_\theta/2$. 

\subsubsection{Proof of Lemma~\ref{lemma:sequential-test-4}}
\label{proof:sequential-test-4}
The starting point of the proof of this statement is the regret bound for the mixture method~(see, for example,~\cite[Chapter 4]{hazan2016introduction}), which tells us that 
\begin{align}
    \log W_n &=\log \lp \int_{-1}^0 \lp \prod_{i=1}^n (1 + \lambda Y_i) \rp d\lambda  \rp \geq  \sup_{\lambda \in [-1/2,0]}  \sum_{i=1}^n\log (1 + \lambda Y_i  ) - 2 \log n,  \label{eq:regret_bound_for_mixture_mdethod}
\end{align}
Note that we have also used the fact that by restricting the domain of $\lambda$ from $[-1,0]$ to $[-1/2, 0]$, we can only reduce the value of the supremum. 
Now, we can use the lower bound $\log(1 + x) \geq x - x^2$ for $x \geq -0.5$ to conclude that 
\begin{align}
    \frac{\log W_n }{n} + \frac{2 \log n}{n} \geq \sup_{\lambda \in [-1/2,0]} \lambda  \lp \frac 1 n \sum_{i=1}^n Y_i \rp - \lambda^2 \lp \frac 1 n \sum_{i=1}^n Y_i^2 \rp \geq 
    \sup_{\lambda \in [-1/2,0]} \lambda  \lp \frac 1 n \sum_{i=1}^n Y_i \rp - \lambda^2.  
\end{align}
This quadratic lower bound is maximized at $\lambda^* = (1/2) \frac{1}{n} \sum_{i=1}^n Y_i$, which for $n \geq N_0$ lies in the set $[-1/2,  -\Delta_\theta/4] \subset [-1/2, 0]$ by Lemma~\ref{lemma:sequential-test-3}. Thus, the optimizing value $\lambda^*$ lies in the valid domain, and plugging this into the lower bound expression, we get 
\begin{align}
    \frac{\log W_n }{n} + \frac{2 \log n}{n} \geq \frac{1}{4} \lp \frac 1 n \sum_{i=1}^n Y_i \rp^2 \geq \frac{\Delta^2_\theta}{16}.  \label{eq:sequential-test-proof-6}
\end{align}
The last inequality again uses the fact that  for $n \geq N_0$, and under the event $E_n \cap F_n$, we know from Lemma~\ref{lemma:sequential-test-3} that $(1/n) \sum_{i=1}^n Y_i \leq - \Delta_\theta/2$, which implies that $\lp \frac 1 n \sum_{i=1}^n Y_i\rp^2 \geq \Delta^2_\theta/4$. On rearranging this,  we get 
\begin{align}
    \frac{\log W_n}{n} \geq \frac{\Delta^2_\theta}{16} - \frac{2\log n}{n}, \quad \text{for} \quad n \geq N_0. 
\end{align}
Finally, let us define $N_1 = \inf \{n \geq 1: (2 \log n)/n \leq \Delta^2_\theta/32\}$, and conclude that for $n \geq N_1\vee N_0$, we have the required
\begin{align}
    \frac{\log W_n}{n} \geq \frac{\Delta^2_\theta}{32}. 
\end{align}
\section{Proof of Proposition~\ref{theorem:identification}}\label{proof:identification}
In this proof, we show that the indicator $\hat{j}_t=\argmax_{j\in\calL}\hat{p}_{t-1}[j]$, where  $\hat{p}_{t-1}[j]=\frac{1}{t-1}\sum_{i=1}^{t-1}\boldsymbol{1}_{g(X_i)=j}$, eventually identify the true underlying distribution index. Since we have $\mathbb{E}_\theta[\mathbf{1}_{g(X)=j}]=\mathbb{P}_\theta(g(X)=j)$ for any $j\in\calL$ and $\theta\in\calL$ and $\{g(X_i)\}_{i\ge1}$ is a \emph{i.i.d.} random sequence, we can apply Hoeffding's inequality to show the desired result. For any $\theta\in\calL$,
\begin{align}
    \mathbb{P}_\theta\lp \hat{j}_t\neq\theta \rp&=\mathbb{P}_\theta\lp \exists m\in\calL\setminus\{\theta\}\,s.t.\, \frac{1}{t-1}\sum_{i=1}^{t-1}\boldsymbol{1}_{g(X_i)=m}>\frac{1}{t-1}\sum_{i=1}^{t-1}\boldsymbol{1}_{g(X_i)=\theta}\rp\\
    &\leq\sum_{m\in\calL\setminus\{\theta\}}\mathbb{P}_\theta\lp \sum_{i=1}^{t-1}\lp \boldsymbol{1}_{g(X_i)=m}-\boldsymbol{1}_{g(X_i)=\theta} \rp >0 \rp,
\end{align}
where the inequality is from the union bound.
Then, by defining
\[
a_i:=\lp \boldsymbol{1}_{g(X_i)=m}-\boldsymbol{1}_{g(X_i)=\theta} \rp - \lp \mathbb{P}_\theta\lp g(X_i)=m\rp - \mathbb{P}_\theta\lp g(X_i)=\theta\rp \rp\qtext{for each} i\ge1,
\]
we have 
\begin{align}
    \mathbb{P}_\theta\lp \hat{j}_t\neq\theta \rp&\leq\sum_{m\in\calL\setminus\{\theta\}}\mathbb{P}_\theta\lp\sum_{i=1}^{t-1} a_i>-\sum_{i=1}^{t-1}\lp\mathbb{P}_\theta\lp g(X_i)=m\rp - \mathbb{P}_\theta\lp g(X_i)=\theta\rp\rp \rp\\
    &=\sum_{m\in\calL\setminus\{\theta\}}\mathbb{P}_\theta\lp\sum_{i=1}^{t-1}a_i>(t-1)\Delta_{\theta,m}\rp,
\end{align}
where $\Delta_{\theta,m}:=\mathbb{P}_\theta\lp g(X_i)=\theta\rp-\mathbb{P}_\theta\lp g(X_i)=m\rp=\mathbb{P}_\theta\lp g(X)=\theta\rp-\mathbb{P}_\theta\lp g(X)=m\rp=\bp_\theta[\theta]-\bp_{\theta}[m]>0$, for any $m\in\calL\setminus\{\theta\}$ by the fact that $g$ is a separable classifier for the given $\mathbf{P}\in\mathcal{P}_{\text{sep}}$.

Then, since $-2\leq a_i\leq2$, by Hoeffding's inequality, we have 
\begin{align}
    \mathbb{P}_\theta\lp\hat{j}_t\neq\theta\rp\leq\exp\lp-\frac{(\Delta_{\theta,m}(t-1))^2}{4t}\rp\in O(e^{-t}).
\end{align}
Hence, by the Borel–Cantelli lemma, one can obtain the implication described in Proposition~\ref{theorem:identification}.

The proof of Corollary~\ref{cor:indication} is provided as follows: 
Applying Proposition~\ref{theorem:identification}, we can bound the misidentification probability under the alternative at the stopping time. Fix $\theta\in[L]$ and any integer $K\ge 1$, 
\begin{align}
\mathbb{P}_\theta(\hat{j}_{\tau}\neq \theta)
&\le \mathbb{P}_\theta({\tau}\ge K,\ \hat{j}_{\tau}\neq \theta)+\mathbb{P}_\theta(\tau< K) \label{eq:misid_split}\\
&\le \mathbb{P}_\theta(\exists n\ge K,\,\hat{j}_n\neq \theta)+\mathbb{P}_\theta(\tau< K) \notag\\
&\le \sum_{n=K}^{\infty}\mathbb{P}_\theta(\hat{j}_n\neq \theta)
\;+\;\sum_{n=1}^{K-1}\mathbb{P}_\theta\!\left(W_n\ge \frac{1}{\alpha}\right), \label{eq:misid_union}
\end{align}
where the last inequality uses a union bound together with the definition $\tau=\inf\{n\ge 1: W_n\ge 1/\alpha\}$.

The first summation decays exponentially in $K$ by Proposition~\ref{theorem:identification}, i.e., $\sum_{n=K}^{\infty}\mathbb{P}_\theta(\hat{j}_n\neq \theta)\in O(e^{-K})$. For the second term, by the fact that $W_n\ge 0$ for any $n\ge1$, Markov's inequality yields $\mathbb{P}_\theta\left(W_n\ge \frac{1}{\alpha}\right)
\le
\alpha\,\mathbb{E}_\theta[W_n]$,
and hence
\[
\sum_{n=1}^{K-1}\mathbb{P}_\theta\!\left(W_n\ge \frac{1}{\alpha}\right)
\;\le\;
\alpha\sum_{n=1}^{K-1}\mathbb{E}_\theta[W_n].
\]
Using the fact that $\mathbb{E}_\theta[W_n]=O(2^n)$, we obtain
\[
\alpha\sum_{n=1}^{K-1}\mathbb{E}_\theta[W_n]
\;\le\;
\alpha\sum_{n=1}^{K-1}O(2^n)
\;=\;
O(\alpha\,K2^K).
\]
Combining the above bounds gives
\[
\mathbb{P}_\theta(\hat{j}_\tau=\theta)
=1-\mathbb{P}_\theta(\hat{j}_\tau\neq\theta)
\;\ge\;
1-O(e^{-K})-O(\alpha\,K2^K).
\]
In particular, choosing $K\in \Theta(\log_2(1/\sqrt{\alpha}))$ implies $\alpha\,K2^K\to 0$ and $e^{-K}\to 0$ as $\alpha\to 0$. 

\section{Proof of Proposition~\ref{prop:veri_assumption}}\label{proof:assumption}
In this proof, for any ground truth distribution tuple $\mathbf{P}\in\mathcal{P}_{
\text{sep}}$, we show that the ERM algorithm provided in Proposition~\ref{prop:veri_assumption} finds a classifier $\hat{g}\in\calG$ such that it is a separable classifier for $\mathbf{P}=\{P_0,\ldots,P_L\}$.

Before proving this result, we recall the notations used in Proposition~\ref{prop:veri_assumption} as follows: for any $g\in\calG$, $\vartheta,m\in\calL$,
\[
\hat{p}_{\vartheta,m}(g)
\defined
\frac{1}{N}\sum_{i=1}^{N}\mathbbm{1}\{g(T_{\vartheta,i})=m\},
\qquad
p_{\vartheta,m}(g)
\defined
\mathbb{P}_{X\sim P_\vartheta}\bigl(g(X)=m\bigr),
\]
where $T_{\vartheta,i}$ is the $i$-th training sample drawn from $P_\vartheta$.
Note that $p_{\vartheta,m}(g)$ coincides with $\bp_\vartheta[m]$ used in Section~\ref{sec:proposed_test}, but we represent it as a function of $g$ to describe the ERM algorithm.

Let us fix any distribution tuple $\mathbf{P}\in\mathcal{P}_{\text{sep}}$. From the definition of $\mathcal{P}_{\text{sep}}$, there exists a classifier $g^*\in\calG$ such that $p_{\vartheta,\vartheta}>\max_{m\in\calL\setminus\{\vartheta\}}p_{\vartheta,m}$ for any $\vartheta\in\calL$. Therefore, we have, for every $\vartheta\in\calL$ and $m\in\calL\setminus\{\vartheta\}$,
\begin{equation}
    p_{\vartheta,\vartheta}(g^*) \;>\; p_{\vartheta,m}(g^*) + 4\gamma\qtext{where}\gamma \defined \frac{\Delta^*}{8}\qtext{and}\Delta^*
\defined
\min_{\vartheta\in\calL}\min_{m\in\calL\setminus\{\vartheta\}}\Bigl(p_{\vartheta,\vartheta}(g^*)-p_{\vartheta,m}(g^*)\Bigr).\label{eq:ve_assump_1}
\end{equation}

Then, we construct our ERM algorithm as follows: 
\[
\hat{g}:=\argmax_{g\in\mathcal{G}}\hat{\Delta}(g)\qtext{where}\hat{\Delta}(g):=\min_{\vartheta\in\mathcal{L}}\min_{m\in\calL\setminus\{\vartheta\}}\left(\hat{p}_{\vartheta,\vartheta}(g) - \hat{p}_{\vartheta,m}(g)\right).
\]
We show that for any $\delta\in(0,1)$, when $N>\max\{d,\frac{8}{\gamma^2}\log\frac{(L+1)^24(e2N)^d/d^d}{\delta}\}$,
\[
\mathbb{P}\left(\forall\,\vartheta\in\mathcal{L},\,p_{\vartheta,\vartheta}(\hat{g})>\max_{m\in\calL\setminus\{\vartheta\}}p_{\vartheta,m}(\hat{g})\right)\geq 1-\delta.
\]

First of all, let $F_m=\{f_g:x\rightarrow\mathbf{1}_{g(x)=m}\,|\,g\in\mathcal{G}\}$ be a collection of binary functions over all $g\in\calG$. We have, for any $\vartheta,m\in\mathcal{L}$ and $\delta\in(0,1)$,
\begin{align}
    \mathbb{P}\left(\sup_{g\in \mathcal{G}}\vert\hat{p}_{\vartheta,m}(g)-p_{\vartheta,m}(g)\rvert>\gamma \right)&=\mathbb{P}\left(\sup_{f_g\in F_m}\lvert \frac{1}{N}\sum_{i=1}^{N}f_g(T_{\vartheta,i})-\mathbb{E}\left[\frac{1}{N}\sum_{i=1}^{N}f_g(T_{\vartheta,i})\right]\rvert>\gamma\right)\\
    &\leq\delta\qtext{if}N>\frac{8}{\gamma^2}\log\frac{4s(F_m,2N)}{\delta},\label{eq:ve_assumption_5}
\end{align}
where the probability and the expectation are taken under ${T_{\vartheta,i}\overset{i.i.d.}{\sim}P_\vartheta}$ and the last inequality is by Vapnik–Chervonenkis theorem \cite[Theorem 22.18]{wasserman2004all} with the finite growth function $s(F_m,2N)$. Specifically, according to Section~\ref{sec:problem_form}, we assume that the complexity of  $\mathcal{G}$ is bounded, i.e., the VC dimension of $F_m$ is upper bounded by a constant $d$ for any $m\in\calL$. Therefore, by Sauer's theorem, we have $s(F_m, 2N)\leq (\frac{e2N}{d})^d$ for all $N\geq d$. %

Then, let $\mathcal{A}:=\{\forall\,\vartheta,m\in\mathcal{L},\sup_{g\in \mathcal{G}}\vert\hat{p}_{\vartheta,m}(g)-p_{\vartheta,m}(g)\rvert\leq\gamma\}$. We have 
\begin{align}
    \mathbb{P}\left(\forall\,\vartheta\in\mathcal{L},\,p_{\vartheta,\vartheta}(\hat{g})>\max_{m\in\calL\setminus\{\vartheta\}}p_{\vartheta,m}(\hat{g})\right)\geq \mathbb{P}\left(\forall\,\vartheta\in\mathcal{L},\,p_{\vartheta,\vartheta}(\hat{g})>\max_{m\in\calL\setminus\{\vartheta\}}p_{\vartheta,m}(\hat{g})\vert\mathcal{A}\right)\mathbb{P}(\mathcal{A}).
\end{align}
Under the event $\mathcal{A}$, we have, $\forall\,g\in\mathcal{G}$,
\begin{align}
    p_{\vartheta,\vartheta}(g)-p_{\vartheta,m}(g)&\geq \hat{p}_{\vartheta,\vartheta}(g)-\gamma -(\hat{p}_{\vartheta,m}(g)+\gamma)\\
    &=\hat{p}_{\vartheta,\vartheta}(g)-\hat{p}_{\vartheta,m}(g)-2\gamma \quad\forall\, \vartheta,m\in\mathcal{L}. \label{eq:ve_assumption_2}
\end{align}
Also, we have the opposite direction, $\forall\,g\in\mathcal{G}$,
\begin{align}
    \hat{p}_{\vartheta,\vartheta}(g)-\hat{p}_{\vartheta,m}(g)\geq p_{\vartheta,\vartheta}(g)-p_{\vartheta,m}(g) - 2\gamma\quad\forall\, \vartheta,m\in\mathcal{L}. \label{eq:ve_assumption_3}
\end{align}
Hence, we get, for any $\vartheta\in\calL$, $m\in\calL\setminus\{\vartheta\}$,
\begin{align}
     \hat{\Delta}(g^*)=\hat{p}_{\vartheta,\vartheta}(g^*)-\hat{p}_{y,m}(g^*)\geq p_{\vartheta,\vartheta}(g^*)-p_{\vartheta,m}(g^*) - 2\gamma>2\gamma, \label{eq:ve_assumption_4}
\end{align}
where the last inequality is due to \eqref{eq:ve_assump_1}.

Then, we have, for any $\vartheta\in\calL$,
\begin{align}
    p_{\vartheta,\vartheta}(\hat g)-p_{\vartheta,m}(\hat g)\geq \hat p_{\vartheta,\vartheta}(\hat g)-\hat p_{\vartheta,m}(\hat g) - 2\gamma=\hat{\Delta}(\hat{g})-2\gamma>2\gamma-2\gamma=0\quad\forall\, m\in \calL\setminus\{\vartheta\},
\end{align}
where the first inequality uses \eqref{eq:ve_assumption_2}, the first equality uses the definition of $\hat{\Delta}$, and the second inequality uses the fact that $\hat{\Delta}(\hat{g})\geq\hat{\Delta}(g^*)$ and \eqref{eq:ve_assumption_4}.
Therefore, we have $\mathbb{P}\left(\forall\,\vartheta\in\mathcal{L},\,p_{\vartheta,\vartheta}(\hat{g})>\max_{m\in\calL\setminus\{\vartheta\}}p_{\vartheta,m}(\hat{g})\vert\mathcal{A}\right)=1$. 

Then, by $\mathbb{P}(\mathcal{A})\geq 1- \mathbb{P}(\mathcal{A}^c)$ and \eqref{eq:ve_assumption_5}
\begin{align}
    \mathbb{P}\left(\exists \vartheta,m\in\calL,\,\sup_{g\in \mathcal{G}}\vert\hat{p}_{\vartheta,m}(g)-p_{\vartheta,m}(g)\rvert>\gamma \right)\leq (L+1)^2\delta \qtext{if}N>\frac{8}{\gamma^2}\log\frac{4s(F_m,2N)}{\delta},
\end{align}
let $\delta=\delta'/(L+1)^2$ for any $\delta'\in(0,1)$, leading to
\[
\mathbb{P}\left(\forall\,\vartheta\in\mathcal{L},\,p_{\vartheta,\vartheta}(\hat{g})>\max_{m\in\calL\setminus\{\vartheta\}}p_{\vartheta,m}(\hat{g})\right)\geq 1-\delta'\qquad\forall\, N\geq \max\{d,\frac{8}{\gamma^2}\log\frac{(L+1)^24(e2N)^d/d^d}{\delta'}\}.
\]
By a simple algebra, we can obtain a simple sufficient condition that $N \ge \max\{d, \frac{16d}{\gamma^2} \log\left(\frac{16e}{\gamma^2}\right) + \frac{16}{\gamma^2} \log\left(\frac{4(L+1)^2}{\delta'}\right)\}$. Hence, as $\delta'\rightarrow0$, since $d, \gamma, L$ are all constants, we get $N\in\Omega\left(\frac{1}{\gamma^2}\log\left(\frac{L}{\delta'}\right)\right)$.

We end this section by providing the formal version of Proposition~\ref{prop:veri_assumption}.
\begin{proposition}[Empirical attainability (\textbf{Formal})]
For any $\mathbf{P} \in \mathcal{P}_{\text{sep}}$, there exists $g^*\in\mathcal{G}$ such that $\Delta^*
\defined
\min_{\vartheta\in\calL}\min_{m\in\calL\setminus\{\vartheta\}}\Bigl(p_{\vartheta,\vartheta}(g^*)-p_{\vartheta,m}(g^*)\Bigr)
>0$ by the definition of $\mathcal{P}_{\text{sep}}$.
Let $\gamma\defined \Delta^*/8$, $F_m:=\{f_g:x\rightarrow\mathbf{1}_{g(x)=m}\,|\,g\in\mathcal{G}\}$ for any $m\in\calL$. Suppose for any $\delta\in(0,1)$,
$N\geq \max\left\{d,\frac{8}{\gamma^2}\log\frac{(L+1)^24(e2N)^d/d^d}{\delta}\right\}$ and  $\text{VCdim}(F_m)\leq d$ for any $m\in\calL$.
Then,  with probability at least $1-\delta$, the ERM solution $\hat g$ satisfies $p_{\vartheta,\vartheta}(\hat g)>\max_{m\in\calL\setminus\{\vartheta\}}p_{\vartheta,m}(\hat g)$ for any $\vartheta\in\calL$. 
\end{proposition}

\section{Proof of Proposition~\ref{prop:lower_bound_of_N}}\label{proof:lower_bound_of_N}
In this proof, we establish a necessary condition for the training sample size $N$ considering the family of level-$\alpha$, power-one tests (consisting of the training function $\phi$ and the stopping rule $\tau$). Recall that we assume, for any test in that family, if the test achieves the criteria for a $\mathbf{P}\in\mathcal{P}_{\text{sep}}$, it also achieves the criteria for any permutation of $\mathbf{P}$. Note that the ``test'' mentioned in this section represents the ``overall'' test, which consists of the training function and the stopping rule. %

Without loss of generality, we pick two distributions $P_m$ and $P_\theta$ (for any $m,\theta\in\mathcal{L}$ and $m\neq \theta$) in a given tuple $\mathbf{P}\in\mathcal{P}_{\text{sep}}$. Let $\mathbf{P}=\{P_m,P_\theta,\ldots,P_L\}$ and $\mathbf{P}'=\{P_\theta,P_m,\ldots,P_L\}$, where the positions of $P_m$ and $P_\theta$ are swapped. Suppose the alternative distribution is drawn from the second distribution in both tuples. That is, under $\mathbf{P}$, we have $H_0:X_t\sim P_m$ and $H_1:X_t\sim P_\theta$, while under $\mathbf{P}'$, we have $H_0: X_t \sim P_\theta$ and $H_1: X_t \sim P_m$. Moreover, from the viewpoint of the classifier, the classifier trained under $\mathbf{P}$ (resp., $\mathbf{P}'$) assigns label $0$ to samples generated by $P_m$ (resp., $P_\theta$), and label $1$ to samples generated by $P_\theta$ (resp., $P_m$).

Then, by the level-$\alpha$, power-one criterion and the assumption mentioned above, $\pi$ must satisfy
\begin{align}\label{eq:level_power_binary}
&\mathbb{P}_{\mathbf{P},0}(\tau<\infty)\le \alpha\qtext{and}
\mathbb{P}_{\mathbf{P},1}(\tau<\infty)=1\\
&\mathbb{P}_{\mathbf{P}',0}(\tau<\infty)\le \alpha\qtext{and}
\mathbb{P}_{\mathbf{P}',1}(\tau<\infty)=1
\end{align}

The idea of the proof is connecting the distinction between these ``two worlds'' and the training sample size. 
Then, we define $A = \{\tau < \infty\}$ and apply Data Processing Inequality (DPI) \cite[Chapter 3]{Polyanskiy_Wu_2025} to get
\begin{align}
    D(\mathbb{P}_{\mathbf{P}',1} \| \mathbb{P}_{\mathbf{P},0}) &\ge d(\mathbb{P}_{\mathbf{P}',1}(A) \| \mathbb{P}_{\mathbf{P},0}(A))\\
    &= 1 \cdot \log \frac{1}{\mathbb{P}_{\mathbf{P},0}(A)} + 0 \cdot \log \frac{0}{1 - \mathbb{P}_{\mathbf{P},0}(A)}\\
    &\geq\log \frac{1}{\alpha}
\end{align}
where $d(p \| q)$ is the binary KL divergence, 
the equality is due to the power-one property $\mathbb{P}_{\mathbf{P}',1} (A)= 1$, and the last inequality uses the level-$\alpha$ property $\mathbb{P}_{\mathbf{P},0}(A) \le \alpha$.

Then, by the definition of $\mathbb{P}_{\mathbf{P}',1}$ and $\mathbb{P}_{\mathbf{P},0}$, we have
\begin{align}
    D(\mathbb{P}_{\mathbf{P}',1} \| \mathbb{P}_{\mathbf{P},0}) =  N (D(P_\theta \| P_m) + D(P_m \| P_\theta)) = N \cdot 2 \mathrm{J}(P_m, P_\theta),
\end{align}
where $\mathrm{J}(P_m, P_\theta):=\frac{1}{2}(D(P_m\|P_\theta)+D(P_\theta\|P_m))$. The first equality follows from the chain rule of KL-divergence and the fact that the training samples are drawn i.i.d. Additionally, under both $\mathbf{P}$ (assuming $H_0$ is true) and $\mathbf{P}'$ (assuming $H_1$ is true), the testing samples are drawn from $P_m$.

Consequently, we have
\begin{equation}
    N \cdot 2 \mathrm{J}(P_m, P_\theta) \ge \log \frac{1}{\alpha} \implies N \ge \frac{\log \frac{1}{\alpha}}{2 \mathrm{J}(P_m, P_\theta)}.
\end{equation}
Since the above inequality holds for any $\theta\neq m$, we get 
\begin{align}
    N\ge \frac{\log \frac{1}{\alpha}}{2\min_{\theta\neq m} \mathrm{J}(P_m, P_\theta)}.
\end{align}

\section{Proof of Theorem~\ref{thm:type-II_lower_bound}}\label{proof:type-II_lower_bound}
In this section, we prove the lower bound of $\Psi\defined\inf_{\pi\in\Pi_{\alpha}(\mathcal{P}_{\text{sep}})}\sup_{\mathbf{P}\in\mathcal{\mathcal{P}_{\text{sep}}}}\mathbb{P}^\pi_{\mathbf{P},\theta}(\tau > n)$ described in Theorem~\ref{thm:type-II_lower_bound}, where $\Pi_{\alpha}(\mathcal{P}_{\text{sep}})$ is the uniformly achievable family of tests defined in Section~\ref{sec:minimax} and $\mathcal{P}_{\text{sep}}$ is the distinguishable family of distributions defined in Section~\ref{sec:problem_form}. We assume that both training and testing data are identical, i.e., no mismatch.

Note that in this section, the classifier $g$ is not fixed as described in Section~\ref{sec:proposed_test}. To study the minimax lower bound over all learning policies and the impact of the classifier, we account for its randomness. Specifically, $g$ depends on the training tuple $\{T^N_\vartheta\}_{\vartheta\in\calL}$ because it is the output of the training policy, i.e., $\phi(\{T^N_\vartheta, \vartheta\}_{\vartheta\in\calL})=g\in\calG$. Moreover, we restrict our attention to a \emph{classifier-based} sequential test. This means the decision-maker determines stopping based on the classifier's output rather than the raw data during the online phase. Therefore, throughout this section, the stopping event $\{\tau\leq t\}$ is adapted to a \emph{smaller} filtration denoted as $\tilde{\mathcal{F}}_{t} \defined \{\{T^N_\vartheta \}_{\vartheta\in\calL}, \{g(X_s)\}_{s\le t}\}$ for any $t\ge1$.

Then, we recall the uniformly achievable family of such classifier-based sequential tests over $\mathcal{P}_{\text{sep}}$ as follows: 
\[
\Pi_\alpha(\mathcal{P}_{\text{sep}}) \coloneqq \left\{ \pi : \sup_{\mathbf{P} \in \mathcal{P}_{\text{sep}}} \mathbb{P}^{\pi}_{\mathbf{P}, 0}(\tau<\infty) \leq \alpha \qtext{and} \inf_{\mathbf{P} \in \mathcal{P}_{\text{sep}}} \mathbb{P}^{\pi}_{\mathbf{P}, \theta}(\tau<\infty) = 1 \ \ \forall \theta \in [L] \right\}.
\]
Our goal is to demonstrate how the training sample size $N$, the family of classifiers $\calG$, and $n$ impact the lower bound of $\Psi\defined\inf_{\pi\in\Pi_{\alpha}(\mathcal{P}_{\text{sep}})}\sup_{\mathbf{P}\in\mathcal{P}_{\text{sep}}}\mathbb{P}^\pi_{\mathbf{P},\theta}(\tau > n)$. The main idea is to utilize Le Cam's two-point method \cite[Chapter 31]{Polyanskiy_Wu_2025}, a useful technique for proving minimax lower bounds from an information-theoretic viewpoint. We start by applying a decomposition used in the converse of \cite[Theorem 2]{Hsu25} and the data processing inequality (DPI) \cite[Chapter 3]{Polyanskiy_Wu_2025} to obtain a general lower bound for $\mathbb{P}^\pi_{\mathbf{P},\theta}(\tau>n)$ for any $\theta\in[L]$. Then, taking the infimum and supremum of this bound, we apply Le Cam's two-point method to introduce the impact of $N$ and $\calG$ simultaneously. Finally, using Assumption~\ref{assum:classifier_capacity_limit}, we conclude the proof. Note that we will describe the formal statement of Assumption~\ref{assum:classifier_capacity_limit} in the following proof.

\paragraph{Step 1: Decomposition and DPI}
Consider any $\mathbf{P}\in\mathcal{P}_{\text{sep}}$ and $\pi\in\Pi_\alpha(\mathcal{P}_{\text{sep}})$, where the null distribution is $P_0$ and the alternative distribution is $P_\theta$ for any $\theta\in[L]$. We then apply a variant decomposition provided in the converse proof of \cite[Theorem 2]{Hsu25}. Specifically, let $A=\{\tau>n\}$, then we have 
\begin{align}
    d(\mathbb{P}^{\pi}_{\mathbf{P},0}(A)\Vert\mathbb{P}^{\pi}_{\mathbf{P},\theta}(A))&=\mathbb{P}^{\pi}_{\mathbf{P},0}(A)\log\frac{\mathbb{P}^{\pi}_{\mathbf{P},0}(A)}{\mathbb{P}^{\pi}_{\mathbf{P},\theta}(A)}+\mathbb{P}^{\pi}_{\mathbf{P},0}(A^c)\log\frac{\mathbb{P}^{\pi}_{\mathbf{P},0}(A^c)}{\mathbb{P}^{\pi}_{\mathbf{P},\theta}(A^c)}\\
    &=-h_b(\mathbb{P}^{\pi}_{\mathbf{P},0}(A))-\mathbb{P}^{\pi}_{\mathbf{P},0}(A)\log\mathbb{P}^{\pi}_{\mathbf{P},\theta}(A)-\mathbb{P}^{\pi}_{\mathbf{P},0}(A^c)\log\mathbb{P}^{\pi}_{\mathbf{P},\theta}(A^c).
\end{align}
where $d(p\Vert q)=p\log\frac{p}{q}+(1-p)\log\frac{1-p}{1-q}$ and $h_b(p)=p\log1/p+(1-p)\log(1/(1-p))$. Therefore, by rearrangement, we have
\begin{align}
    \log\frac{1}{\mathbb{P}^\pi_{\mathbf{P},\theta}(A)}=\frac{1}{\mathbb{P}^{\pi}_{\mathbf{P},0}(A)}\left(d(\mathbb{P}^{\pi}_{\mathbf{P},0}(A)\Vert\mathbb{P}^{\pi}_{\mathbf{P},\theta}(A))+h_b(\mathbb{P}^{\pi}_{\mathbf{P},0}(A))+\mathbb{P}^{\pi}_{\mathbf{P},0}(A^c)\log\mathbb{P}^{\pi}_{\mathbf{P},\theta}(A^c)\right).
\end{align}
By the level-$\alpha$ property, we have $\mathbb{P}^{\pi}_{\mathbf{P},0}(A)=1-\mathbb{P}^{\pi}_{\mathbf{P},0}(\tau\leq n)\geq 1-\alpha$. Moreover, since $A\in\tilde {\mathcal{F}}_n$, by data processing (applied to the map $\omega\mapsto \mathbf 1\{\tau(\omega)>n\}$),
we have $d(\mathbb P^\pi_{\mathbf{P},0}(A)\Vert \mathbb P^\pi_{\mathbf{P},\theta}(A))\le D(\mathbb P^\pi_{\mathbf{P},0}\mid_{ \tilde{\mathcal{F}}_n}\Vert \mathbb P^\pi_{\mathbf{P},\theta}\mid_{\tilde{\mathcal{F}}_n})$, where $\mathbb P^\pi_{\mathbf{P},0}\mid_{ \tilde{\mathcal{F}}_n}$ denotes the joint distribution of the tuple $\{\{T_\vartheta^N\}_{\vartheta\in\calL},\{g(X_i)\}_{i=1}^{N}\}$. Hence, we get 
\begin{align}
    \log\frac{1}{\mathbb{P}^\pi_{\mathbf{P},\theta}(A)}\leq\frac{1}{1-\alpha}\left(D(\mathbb P^\pi_{\mathbf{P},0}\mid_{\tilde{\mathcal{F}}_n}\Vert \mathbb P^\pi_{\mathbf{P},\theta}\mid_{\tilde{\mathcal{F}}_n}))+\log(2)\right),
\end{align}
where the inequality uses the fact that $h_b(p)\leq \log(2)$ and $\log\mathbb{P}^\pi_{\mathbf{P},\theta}(A^c)\leq0$.

Before proceeding with the proof, let $(\bp_m^g)^{\otimes n}$ denote the joint distribution of $\{g, g(X_i)\}_{i=1}^{n}$ where $X_i\simiid P_m$ for any $m\in\calL$ and $g=\phi(\{T^N_\vartheta,\vartheta\}_{\vartheta\in\calL})$. Given $g$, if $m\in\calL$ is the ground truth distribution, these random variables are drawn i.i.d. from the distribution of $g(X)$ where $X\sim P_m$ (denoted as $\bp_m$ defined in Section~\ref{sec:proposed_test}). In this section, to emphasize the role of the classifier, we let $\bp^g_m\equiv\bp_m$ under a given classifier $g$. Consequently, we have $(\bp_m^g)^{\otimes n}=\Pi_{i=1}^{n} \bp^g_m$ if $g$ is given. Then, we provide an upper bound for the KL-divergence term. Denote $\mathbf{T}^N:=\{T_\vartheta^N\}_{\vartheta\in\calL}$.
\begin{align}
    D(\mathbb P^\pi_{\mathbf{P},0}\mid_{\tilde{\mathcal{F}}_n}\Vert \mathbb P^\pi_{\mathbf{P},\theta}\mid_{\tilde{\mathcal{F}}_n})
    &=D(\Pi_{\vartheta=0}^{L}P_\vartheta^{\otimes N}\cdot (\bp_0^g)^{\otimes n}\Vert \Pi_{\vartheta=0}^{L}P_\vartheta^{\otimes N}\cdot (\bp_\theta^g)^{\otimes n})\\
    &= D(\Pi_{\vartheta=0}^{L}P_\vartheta^{\otimes N}\Vert \Pi_{\vartheta=0}^{L}P_\vartheta^{\otimes N}) + D((\bp_0^g)^{\otimes n}\Vert (\bp_\theta^g)^{\otimes n}\,|\,\Pi_{\vartheta=0}^{L}P_\vartheta^{\otimes N})\label{eq:two-stage-lower-bound-1}\\
    &=\mathbb{E}_{\mathbf{T}^N}\left[D\left((\bp_0^g)^{\otimes n|\mathbf{T}^N}\lVert (\bp_\theta^g)^{\otimes n|\mathbf{T}^N}\right)\right]\label{eq:two-stage-lower-bound-2}\\
    &=n\mathbb{E}_{\mathbf{T}^N}\left[D\left(\bp_0^g\lVert \bp_\theta^g\right)\right]\label{eq:two-stage-lower-bound-3}
\end{align}
where \eqref{eq:two-stage-lower-bound-1} uses the chain rule, \eqref{eq:two-stage-lower-bound-2} is since the first term is zero and the definition of the conditional KL-divergence (note that $(\bp_0^g)^{\otimes n|\mathbf{T}^N}$ denotes the conditional distribution of $(\bp_0^g)^{\otimes n}$ conditioning on the training samples), and \eqref{eq:two-stage-lower-bound-3} is from our problem formulation that given the training tuple $\mathbf{T}^N$, equivalently given $g$, the output of the classifier of $X_i$ is $g(X_i)\simiid\bp^g_\theta$ for each $i\leq n$.

Therefore, plugging \eqref{eq:two-stage-lower-bound-3} back into the type-II error bound, we get
\begin{align}
    \log\frac{1}{\mathbb{P}^\pi_{\mathbf{P},\theta}(A)}\leq\frac{1}{1-\alpha}\left(n\mathbb{E}_{\mathbf{T}^N}\left[D\left(\bp_0^g\lVert \bp_\theta^g\right)\right]+\log(2)\right).\label{eq:step_1}
\end{align}
Now, we further upper-bound the expectation term to get a general lower bound of $\mathbb{P}^\pi_{\mathbf{P},\theta}(A)$. By $\mathbb{E}_{\mathbf{T}^N}\left[D\left(\bp_0^g\lVert \bp_\theta^g\right)\right]\leq \sup_{g\in\mathcal{G}}D(\bp_0^g\Vert \bp^g_\theta)$, where $\mathcal{G}$ is the family of classifier, we have
\begin{align}
    &\log\frac{1}{\mathbb{P}^\pi_{\mathbf{P},\theta}(A)}\leq\frac{1}{1-\alpha}\left(n \sup_{g\in\mathcal{G}}D(\bp_0^g\Vert \bp_\theta^g)+\log2\right)\\
    &\Rightarrow \mathbb{P}^\pi_{\mathbf{P},\theta}(\tau>n)\geq \exp\left(-\frac{1}{1-\alpha}\left(n \sup_{g\in\mathcal{G}}D(\bp_0^g\Vert \bp_\theta^g)+\log2\right)\right).%
\end{align}
Consequently, we obtain a general lower bound that is valid for all $\theta\in[L]$, $\pi\in\Pi_\alpha(\mathcal{P}_{\text{sep}})$, and $\mathbf{P}\in\mathcal{P}_{\text{sep}}$. Observe that for sufficiently large $n$, the dominant exponent is $\sup_{g\in\mathcal{G}}D(\bp_0^g\Vert \bp_\theta^g)$. Therefore, if $\calG$ is sufficiently rich, one can find a classifier with stronger separability for $P_0$ and $P_\theta$, leading to a smaller lower bound on the type-II error. However, we remark that no matter how rich $\calG$ is, by DPI, $D(\bp_0^g\Vert \bp_\theta^g)\leq D(P_0\|P_\theta)$, which is an intrinsic problem difficulty.

\paragraph{Step 2: Apply Le Cam's two-point method}
Although we have derived the general lower bound, it does not show the impact of the training sample size $N$. This is because we upper bounded the expectation term $\mathbb{E}_{\mathbf{T}^N}\left[D\left(\bp_0^g\lVert \bp_\theta^g\right)\right]$, which encodes the information about the training sample size, by simply taking the supremum over the family of classifiers. We now apply Le Cam's two-point method to introduce the impact of $N$ and $\calG$ simultaneously in the minimax setting.

By applying minimax operation on \eqref{eq:step_1}, we have 
\begin{align}
    \inf_{\pi\in\Pi_\alpha(\mathcal{P}_{\text{sep}})}\sup_{\mathbf{P}\in\mathcal{P}_{\text{sep}}}\mathbb{P}^\pi_{\mathbf{P},\theta}(\tau>n)\geq \exp\left(-\frac{1}{1-\alpha}\left(n \sup_{\pi\in\Pi_\alpha(\mathcal{P}_{\text{sep}})}\inf_{\mathbf{P}\in\mathcal{P}_{\text{sep}}}\mathbb{E}_{\mathbf{T}^N}\left[D\left(\bp_0^g\lVert \bp_\theta^g\right)\right]+\log2\right)\right),
\end{align}
Then, let us provide a non-trivial upper bound for the $\sup\inf$ expectation term. 

Let $D^*(\mathbf{P}):=\sup_{g\in\mathcal{G}}D(\bp_0^g\Vert \bp^g_\theta)$. We can write $\mathbb{E}_{\mathbf{T}^N}\left[D\left(\bp_0^g\lVert \bp_\theta^g\right)\right]= D^*(\mathbf{P})-\left(D^*(\mathbf{P})-\mathbb{E}_{\mathbf{T}^N}\left[D\left(\bp_0^g\lVert \bp_\theta^g\right)\right]\right)$.
Following this, we define the \emph{non-negative loss} (or suboptimality gap) $\mathcal{R}(g,\mathbf{P}):=D^*(\mathbf{P})-D\left(\bp_0^g\lVert \bp_\theta^g\right)$ (note that $g$ depends on the tuple $\mathbf{T}^N$), leading to 
\begin{align}
    \mathbb{E}_{\mathbf{T}^N}\left[D\left(\bp_0^g\lVert \bp_\theta^g\right)\right] = D^*(\mathbf{P})-\mathbb{E}_{\mathbf{T}^N}[\mathcal{R}(g,\mathbf{P})].\label{eq:step2}
\end{align}
Then, we consider two tuples in $\mathcal{P}_{\text{sep}}$ which are parametrized by $\nu_0$ and $\nu_1$, denoted as $\mathbf{P}_{\nu_0}$ and $\mathbf{P}_{\nu_1}$ defined in Section~\ref{sec:minimax}, where $\nu_0$, $\nu_1$ are defined in Theorem~\ref{thm:type-II_lower_bound}, $\nu_1=\nu_0+\delta$ for some $\delta>0$. These two tuples serve as the basic elements in Le Cam's two-point method.

Therefore, from \eqref{eq:step2}, we have 
\begin{align}
    \sup_{\pi\in\Pi_\alpha(\mathcal{P}_{\text{sep}})}\inf_{\mathbf{P}\in\mathcal{P}_{\text{sep}}}\mathbb{E}_{\mathbf{T}^N}\left[D\left(\bp_0^g\lVert \bp_\theta^g\right)\right]
    &\leq \sup_{\pi\in\Pi_\alpha(\mathcal{P}_{\text{sep}})}\min_{\mathbf{P}\in\{\mathbf{P}_{\nu_0}, \mathbf{P}_{\nu_1}\}} \mathbb{E}_{\mathbf{T}^N}\left[D\left(\bp_0^g\lVert \bp_\theta^g\right)\right]\\
    &\equiv\sup_{\pi\in\Pi_\alpha(\mathcal{P}_{\text{sep}})}\min_{\nu\in\{\nu_0,\nu_1\}}\mathbb{E}_{\mathbf{T}^N}\left[D(\bp^g_{0,\nu}\|\bp^g_{\theta,\nu})\right]\\
    &\leq\max_{\nu\in\{\nu_0,\nu_1\}} D^*(\mathbf{P}_\nu)-\inf_{\pi\in\Pi_\alpha(\mathcal{P}_{\text{sep}})}\max_{\nu\in\{\nu_0,\nu_1\}}\mathbb{E}_{\mathbf{T}^N}[\mathcal{R}(g,\mathbf{P}_\nu)],\label{eq:step2_1}
\end{align}
where $\bp^g_{m,\nu}\defined\mathbb{P}_{X\sim P_{m,\nu}}(g(X)=\cdot)$ is the probability mass function of $g(X)$ where $X\sim P_{m,\nu}$ (the $m$-th distribution in the tuple $\mathbf{P}_\nu$) for any $m\in\calL$ and $\nu\in\{\nu_0,\nu_1\}$, and the first inequality uses the fact that taking the infimum over a smaller feasible set is larger.  %

Then, we continue to lower-bound the second term in \eqref{eq:step2_1}. Following the standard argument of Le Cam's two-point method \cite[Chapter 31]{Polyanskiy_Wu_2025}, we get 
\begin{align}
    \inf_{\pi\in\Pi_\alpha(\mathcal{P}_{\text{sep}})}\max_{\nu\in\{\nu_0,\nu_1\}}\mathbb{E}_{\mathbf{T}^N}[\mathcal{R}(g,\mathbf{P}_\nu)]&
    \geq\frac{\mathcal{R}_{\nu_0,\nu_1}}{2}\left(1-\text{TV}\left( \Pi_{i=0}^{L}P^{\otimes N}_{i,\nu_0}, \Pi_{i=0}^{L}P^{\otimes N}_{i,\nu_1}\right)\right)\\
    &\geq\frac{\mathcal{R}_{\nu_0,\nu_1}}{4}\exp\left(-N\sum_{i=0}^{L}D(P_{i,\nu_0}\Vert P_{i,\nu_1})\right)
\end{align}
where $\mathcal{R}_{\nu_0,\nu_1}:=\inf_{g\in\mathcal{G}}\left(\mathcal{R}(g,\mathbf{P}_{\nu_0})+\mathcal{R}(g,\mathbf{P}_{\nu_1})\right)$, the first inequality is from Le Cam's two-point method, and the second inequality is due to
\begin{align}
    \text{TV}\left( \Pi_{i=0}^{L}P^{\otimes N}_{i,\nu_0}, \Pi_{i=0}^{L}P^{\otimes N}_{i,\nu_1}\right)&\leq1-\frac{1}{2}\exp\left(-D\left( \Pi_{\vartheta=0}^{L}P^{\otimes N}_{\vartheta,\nu_0}\Vert \Pi_{\vartheta=0}^{L}P^{\otimes N}_{\vartheta,\nu_1} \right)\right)\\
    &=1-\frac{1}{2}\exp\left(-N\sum_{i=0}^{L}D(P_{i,\nu_0}\Vert P_{i,\nu_1})\right),
\end{align}
in which the first inequality is due to \cite[Lemma 7]{Duchi2024} and the second equality is due to the independence of training samples. 

Therefore, we get the minimax lower bound: For any $\theta\in[L]$,
\begin{align}
    &\log\inf_{\pi\in\Pi_\alpha(\mathcal{P}_{\text{sep}})}\sup_{\mathbf{P}\in\mathcal{P}_{\text{sep}}}\mathbb{P}^\pi_{\mathbf{P},\theta}(\tau>n)\\
    &\geq -\frac{1}{1-\alpha}\left(n\left(\max_{\nu\in\{\nu_0,\nu_1\}} D^*(\mathbf{P}_\nu)-\frac{\mathcal{R}_{\nu_0,\nu_1}}{4}\exp\left(-N\sum_{i=0}^{L}D(P_{i,\nu_0}\Vert P_{i,\nu_1})\right)\right)+\log2\right).
\end{align}
Now, by the assumption described in Theorem~\ref{thm:type-II_lower_bound}, $D(P_{i,\nu_0}\Vert P_{i,\nu_1})\leq M \delta^2$, $\forall\,i\in\calL$, we have
\begin{align}
    &\log\inf_{\pi\in\Pi_\alpha(\mathcal{P}_{\text{sep}})}\sup_{\mathbf{P}\in\mathcal{P}_{\text{sep}}}\mathbb{P}^\pi_{\mathbf{P},\theta}(\tau>n)\\
    &\geq -\frac{1}{1-\alpha}\left(n\left(\max_{\nu\in\{\nu_0,\nu_1\}} D^*(\mathbf{P}_\nu)-\frac{\mathcal{R}_{\nu_0,\nu_1}}{4}\exp\left(-NM(L+1)\delta^2\right)\right)+\log2\right)\\
    &\geq -\frac{1}{1-\alpha}\left(n\left(\max_{\nu\in\{\nu_0,\nu_1\}} D(P_{0,\nu}\Vert P_{\theta,\nu})-\frac{\mathcal{R}_{\nu_0,\nu_1}}{4}\exp\left(-NM(L+1)\delta^2\right)\right)+\log2\right),
\end{align}
where the last inequality is due to DPI.

Finally, by Assumption~\ref{assum:classifier_capacity_limit} (we provide its formal version below), we can get the result shown in Theorem~\ref{thm:type-II_lower_bound}.
\begin{assumption}[Classifier capacity limit (\textbf{Formal})]\label{assump:classifier_capacity_limit_formal}
    For the considered $\mathbf{P}_{\nu_0}$ and $\mathbf{P}_{\nu_1}$, there exists a bounded constant $B>0$ such that $\mathcal{R}_{\nu_0,\nu_1}:=\inf_{g\in\mathcal{G}}\left(\mathcal{R}(g,\mathbf{P}_{\nu_0})+\mathcal{R}(g,\mathbf{P}_{\nu_1})\right)>B$.
\end{assumption}

\begin{remark}
    The bound, $D(P_{i,\nu_0}\Vert P_{i,\nu_1})\leq M \delta^2$, is a classical result in information geometry \cite[Exercise 11.7]{cover06}, which shows that $D(P_{\theta}\Vert P_\theta')=\frac{1}{2}I_\theta(\theta-\theta')^2+o((\theta-\theta')^2)$ (abusing notation, we assume $\theta\in\mathbb{R}$ here, rather than the distribution index), where $I_\theta:=\mathbb{E}_{X\sim P_\theta} \left[ \frac{\partial}{\partial \theta} \ln P_{\theta}(X) \right]^2$ is the Fisher information. If $|\theta'-\theta|=\delta$ is sufficiently small, we obtain the bound. However, we do not specify the design of $\nu_0$ and $\nu_1$ in this paper, leaving it as future work.
\end{remark}

\section{Proof of Theorem~\ref{theorem:change-detection}}
\label{proof:change-detection}

\paragraph{ARL control.} Observe that the process $\{M_n: n \geq 1\}$ is a Shirayev-Roberts type e-process as defined by~\cite{shin2022detectors}, since each $\{W_n^{(k)}: n \geq k\}$ is an e-process for $P_0$. Hence, as a consequence of~\cite[Theorem 2.4]{shin2022detectors}, we have 
    $\mathbb{E}_\infty[\tau] \geq \frac{1}{\alpha}$, 
when there is no change. 

\paragraph{Detection Delay Bound.} Now suppose that there is a change in the distribution at some time $T$. Let us introduce the stopping time $\tau^{(k)} = \inf \{n-k + 1: W^{(k)}_n \geq 1/\alpha,\; n \geq k\}$. Then, a simple consequence of the definition of $\tau$ is that 
\begin{align}
    \tau \leq \inf_{k \geq 1}\lp \tau^{(k)} + k - 1\rp \leq  \inf_{k \geq T+1}\lp \tau^{(k)} + k - 1\rp \leq \tau^{(T+1)} + T, 
\end{align}
almost surely. Hence, we have 
\begin{align}
    (\tau - T)^+ \leq \tau^{(T+1)}. 
\end{align}
Observe that $\tau^{(T+1)}$ is the same as the sequential test we analyzed under the alternative in~\Cref{theorem:sequential-test}. Under the assumption that the post-change observations $\{X_t: t \geq T+1\}$ are independent of the pre-change observations $\{X_t: t \leq T\}$, we then have, assuming $\theta\in\{1,\ldots,L\}$ is the ground truth parameter of the post-change distribution,
\begin{align}
    \mathbb{E}_{T,\theta}\lb (\tau - T)^+ |\calF_T \rb \leq \mathbb{E}_{T,\theta}[\tau^{(T+1)}\mid\calF_T] = \mathbb{E}_{T,\theta}[\tau^{(T+1)}] = \mathbb{E}_\theta[\tau^{(1)}]. 
\end{align}
The term $E_\theta[\tau^{(1)}]$ satisfies the same upper bound we derived in~\Cref{theorem:sequential-test}, which by the above chain implies the upper bound required for the detection delay by taking a supremum  over $T \in \mathbb{N}$.

\end{document}